\documentclass[english,reqno, 11pt]{amsart}
\usepackage[T1]{fontenc}
\usepackage[latin9]{inputenc}
\usepackage{amsthm}
\usepackage{amssymb}
\usepackage{amsmath}
\usepackage{mathabx}
\usepackage{xcolor}
\usepackage{verbatim, color}
\usepackage{bbm}
\usepackage{float}
\usepackage{dsfont, graphicx}
\usepackage{epstopdf}
\usepackage[hidelinks]{hyperref}
\usepackage[T1]{fontenc}
\usepackage{amssymb,amsmath, amsthm, amsfonts}
\usepackage{graphicx}

\usepackage[hidelinks]{hyperref}

\usepackage{thmtools,thm-restate}

\makeatletter
\numberwithin{equation}{section}
\numberwithin{figure}{section}
\theoremstyle{plain}
\newtheorem{thm}{Theorem}[section]
\newtheorem{prop}[thm]{Proposition}
\newtheorem{definition}[thm]{Definition}

\newtheorem{lem}[thm]{Lemma}

\newtheorem{rem}[thm]{Remark}

  \newcounter{casectr}
  
\theoremstyle{definition}

\theoremstyle{remark}

\newcommand{\RRR}{\mathbb{R}}
\newcommand{\ZZZ}{\mathbb{Z}}

\newcommand{\xls}{\frac{x_{s}}{\lambda}}
\newcommand{\tgs}{\tilde{\gamma}_{s}}
\newcommand{\llsb}{\frac{\lambda_{s}}{\lambda}+b}
\newcommand{\lls}{\frac{\lambda_{s}}{\lambda}}
\newcommand{\tb}{\Theta_{b}}
\newcommand{\sbb}{\Sigma_{b}}

\newcommand{\HHH}{\mathcal{H}}

\newcommand{\JJJ}{\mathcal{J}}

\newcommand{\GGG}{\mathcal{G}}

\newcommand{\errh}{\int |\nabla I_{N\lambda}\epsilon|^{2}}
\newcommand{\errm}{\int \epsilon^{2}e^{-|y|}}

\newcommand{\FFF}{\mathcal{F}}
\newcommand{\FF}{F}

\newcommand{\eo}{\epsilon_{0}}

\newcommand{\qb}{{Q}_{b}}
\newcommand{\qbb}{\widetilde{Q}_{b}}
\newcommand{\zb}{{\zeta}_{b}}
\newcommand{\zbb}{\widetilde{\zeta}_{b}}
\newcommand{\qbo}{\widetilde{Q}_{b_{0}}}
\newcommand{\eez}{\epsilon_{0}}
\newcommand{\eeo}{\epsilon_{1}}
\newcommand{\eet}{\epsilon_{2}}
\newcommand{\aaa}{\alpha}

\newcommand{\inl}{I_{N\lambda}}

\begin{document}
\title[Construction of $L^2$ log-log blowup solutions for the mass critical NLS]{Construction of $L^2$ log-log blowup solutions for the mass critical nonlinear Schr\"odinger equation}
\author[C. Fan]{Chenjie Fan}
\address{Academy of Mathematics and Systems Science, CAS, Beijing, China}
\email{fancj@amss.ac.cn}
\author[D. Mendelson]{Dana Mendelson}
\address{University of Chicago, Chicago, IL}
\email{dana@math.uchicago.edu}
\begin{abstract}
In this article, we study the log-log blowup dynamics for the mass critical  nonlinear Schr\"odinger equation on $\RRR^2$ under rough but structured random perturbations at $L^{2}(\RRR^2)$ regularity. In particular, by employing probabilistic methods, we provide a construction of a family of $L^{2}(\RRR^2)$ regularity solutions which do not lie in any $H^{s}(\RRR^2)$ for any $s>0$, and which blowup according to the log-log dynamics. 
\end{abstract}

\thanks{C.~Fan gratefully acknowledges support from a Simons Travel Grant  and a start up funding from AMSS. D.~Mendelson gratefully acknowledges support from NSF grant DMS-1800697.} 

\maketitle

\section{Introduction}
\subsection{Main results and background}
We consider the focusing cubic nonlinear Schr\"odinger equation (NLS) on $\mathbb{R}^{2}$
\begin{equation}\label{eq:nls}
\begin{cases}
iu_{t}+\Delta u=-|u|^{2}u, \quad (t, x) \in \RRR \times \RRR^2& \\
u(0,x)=u_{0}. &
\end{cases}
\end{equation}
The goal of this article is to construct log-log blowup solutions at $L_{x}^{2}(\RRR^{2})$ regularity  via random data methods.

\medskip
The NLS \eqref{eq:nls} has three conservation laws
\begin{itemize}
\item Mass: $M(u)= \int |u|^2 $,
\item Momentum: $P(u)=\Im \int u\nabla \bar{u} $,
\item Energy: $E(u)=\frac{1}{2}\int |\nabla u|^{2}-\frac{1}{4}|u|^{4}$,
\end{itemize}
and enjoys the scaling, translation and phase symmetries. In particular, if $u$ solves \eqref{eq:nls} with initial data $u_{0}$, then 
\[
\frac{1}{\lambda_{0}}u\biggl(\frac{t}{\lambda_{0}^{2}}, \frac{x-x_{0}}{\lambda_{0}}\biggr)e^{i\gamma_{0}}, \qquad x_0\in \RRR^{2},\,\, \lambda_0>0,\,\, \gamma_0 \in \RRR
\]
 solves \eqref{eq:nls} with initial data 
 \[
 u_{0,\lambda_{0}} = \frac{1}{\lambda_{0}} u_0\biggl(\frac{x-x_{0}}{\lambda_0}\biggr) e^{i \gamma_0}.
 \]
 One may verify that the mass $M(u)$ is invariant under the same scaling symmetry, and hence the equation \eqref{eq:nls} is referred to as the mass critical NLS. We note that solutions of \eqref{eq:nls} also enjoy the so-called Galilean and pesudoconformal symmetries. We will not explicitly use these symmetries in the present work even though we will rely on many previous results on log-log blowup solutions in which these symmetries play a central role in the analysis.

\medskip
It is classical that $L_x^{2}(\RRR^2)$ initial data gives rise to unique local in time solutions of \eqref{eq:nls}, see \cite{cazenave1989some}. The focusing nature of \eqref{eq:nls} implies, in particular, the existence of a ground state solution, $Q(x)$, which is the unique $L^{2}(\RRR^2)$ radial positive solution of
\begin{equation}\label{eq: ground}
-\Delta Q+Q=-Q^{3}.
\end{equation}
The ground state plays an essential role in the blowup behavior of \eqref{eq:nls}, and in particular, it provides a threshold for blowup dynamics in the following sense: for all $L^{2}_x(\RRR^2)$ solutions with mass strictly below $\|Q\|_{L_{x}^{2}}$, the associated flow is global and scatters asymptotically, see the work of Weinstein \cite{weinstein1983nonlinear} and Dodson \cite{dodson2015global}. Moreover, there exists an explicit blowup solution with mass equal to $\|Q\|_{L_{x}^{2}}^{2}$, given by
\begin{equation}\label{eq: sblow}
S(t,x)=\frac{1}{t}Q\left(\frac{x}{t}\right)e^{-i/t+i\frac{|x|^{2}}{4t}},
\end{equation}
which, in some sense, is  the unique minimal mass blowup solution, see \cite{merle1993determination}.  

Classical virial identity arguments due to Glassey \cite{glassey1977blowing} establish the existence of a large family of negative energy blowup solutions, however the argument does not directly characterize the blowup mechanism for such solutions. It is an active area of research to understand blowup for \eqref{eq:nls} from a constructive perspective, so that one may better understand possible blowup mechanisms in general. For blowup solutions with mass slightly above the ground state,
\begin{equation}\label{eq: slight}
\|Q\|_{2}<\|u_{0}\|_{2}<\|Q\|_{2}+\alpha^{*}, 
\end{equation}
where $\alpha^{*}$ is a small universal number, one of the most well understood blowup dynamic is the so-called \emph{log-log blowup}.  Log-log blowup solutions have been studied numerically in \cite{landman1988rate}, and the first mathematical construction of such solutions was provided by Perelman in \cite{perelman2001blow}. These solutions were subsequently systematically studied by Merle and Rapha\"el in \cite{merle2005blow, merle2003sharp, merle2004universality, merle2006sharp}. In particular,  Merle and Rapha\"el prove that for all $H_x^{1}(\RRR^2)$ solutions to \eqref{eq:nls} which have non-positive energy\footnote{The result of Merle and Rapha\"el is more general and  this  nonpositive energy assumptions can be relaxed. However, general positive energy solutions are less understood compared to those blowing up according to the log-log law.} and with mass slightly above ground state, i.e. in the range \eqref{eq: slight}, such solutions will blowup in finite time\footnote{It is already highly nontrivial that such a solution will blow up in finite time.} $T<\infty$,  with precise asymptotics as $t$ approaches the blowup time $T$ given by the following:
    \begin{definition}[Log-log blowup dynamics]\label{def: loglogblow}
\begin{equation}\label{eq: loglogblow}
u(t,x)= \frac{1}{\lambda(t)}(Q+\epsilon)\left(\frac{x-x(t)}{\lambda(t)}\right)e^{-i\gamma(t)}, \quad  \frac{1}{\lambda(t)}\sim \sqrt{\frac{\ln |\ln T-t|}{T-t}}
\end{equation}
and where 
\[
\epsilon(t)\xrightarrow{t\rightarrow T} 0 \textup{ in }\dot{H}^{1}(\RRR^2)\cap L_{loc}^{2}(\RRR^2).
\]
\end{definition}
Such blowup was shown to be stable in $H_x^{1}(\RRR^2)$ in \cite{raphael2005stability},  and was later proved to be stable under $H_x^{s}(\RRR^2)$ perturbation, for all $s>0$ by Colliander and Rapha\"el \cite{colliander2009rough}, though one needs to reformulate the notion of log-log blowup (in a natural way) for infinite energy solutions. It is unclear whether such blowup is stable under $L_x^{2}(\RRR^2)$ perturbation, although one may guess that the answer is negative given the result of \cite{merle2005profiles}. 

\medskip
In light of the speculation that stability of log-log blowup may be false for arbitrary data in $L_x^2(\RRR^2)$, and the fact that this long-standing question remains open, in the current work we investigate the stability of log-log blowup solutions under \emph{random} $L_x^2(\RRR^2)$ perturbations. Beginning with the seminal work of Bourgain \cite{B96}, the behavior of nonlinear dispersive equations with random initial data has been an active field of research, see further discussion in Section \ref{ssec:compare} below. Indeed, in spite of the absence of known deterministic well-posedness theory, or even the existence of ill-posedness results, randomization often lets one establish that a given dispersive equation is well-posed \emph{almost surely} in a particular low-regularity function space.

In the current article we employ randomization for a different and novel purpose. Our aim is to establish the existence of blowup solutions at $L_x^2(\RRR^2)$ regularity via a probabilistic construction. We state our main theorem non-technically for the time being:

\begin{thm}\label{thm: main}
The log-log blowup dynamics of Definition \ref{def: loglogblow} are stable, with high probability, under (certain structured) random $L_x^2(\RRR^2)$ perturbations.
\end{thm}

We will begin with well-prepared $H_x^{1}(\RRR^2)$ data which are known to lead to log-log blowup, and we perturb this initial data with random initial data, constructed as follows: let $\{g_k\}_{k \in \mathbb{Z}^2}$ be a sequence of iid complex Gaussian mean-zero random variables. Let $\{P_k\}_{k \in \mathbb{Z}^2}$ be unit-scale projections to frequency $k \in \mathbb{Z}^2$, defined as the Fourier multiplier with respect to translations of a \emph{fixed} Schwartz function
\begin{align}\label{eq:psi_k}
\psi_k(\xi) := \psi(\xi - k),
\end{align}
that is
\begin{align}\label{eq:p_k}
P_k f = \mathcal{F}^{-1} (\psi_k(\xi) \widehat{f}(\xi)).
\end{align}
We crucially exploit that these Fourier projections satisfy a unit-scale Bernstein inequality, namely for all $1 \leq r_1 \leq r_2 \leq \infty$ we have that 
\begin{align} \label{equ:unit_scale_bernstein}
  \|P_k f\|_{L^{r_2}_x(\RRR^2)} \leq C(r_1, r_2) \|P_k f\|_{L^{r_1}_x(\RRR^2)}
\end{align}
with a constant which is independent of $k \in \mathbb{Z}^4$.

Let $f \in L_x^2(\RRR^2)$, and define its randomization
\begin{align}\label{eq:rand}
f^\omega = \sum_{k  \in \mathbb{Z}^2} g_k(\omega) P_k f = \sum_{k  \in \mathbb{Z}^2} g_k(\omega) (\psi_k \widehat{f} )^{\vee}.
\end{align}
Similar randomizations have previously been used in Euclidean space, first in \cite{ZF}, and subsequently in \cite{LM1}, \cite{BOP1}. One can show that if $f \in L^2_x(\RRR^2) \setminus H^s_x(\RRR^2)$ for some $s > 0$, then $f^\omega \in L_x^2(\RRR^2) \setminus H^s_x(\RRR^2)$ almost surely, and throughout, we will restrict to the subset of full measure of $\Omega$ so that this is indeed the case without further comment.  

In the present application, we will take $f$ to be piecewise constant in Fourier space, i.e. $f_k := P_k f$ constant, and we further require that $f_k$ satisfy\footnote{This is to mimic the randomization in \cite{B96}.}
\begin{align}\label{eq: bourgainsetting}
|f_k| \leq \frac{C}{|k|}, \quad k \geq 1.
\end{align}
Additionally, we normalize
\begin{align} \label{eq: L2sum}
\sum_{k}|f_{k}|^{2}=1.
\end{align}
Note that in particular, there are many $L_x^2(\RRR^2)$ functions $f$ with this property which do not belong to $H^s(\RRR^2)$ for any $s > 0$, and hence $f^{\omega}$  does not belong to $H_x^{s}(\RRR^2)$ for any $s>0$, consider for instance the function $f$ which satisfies
\[
|f_k| \sim \frac{1}{|k| \log^2 |k|} \quad k \in \ZZZ^2.
\]
We note that our result works almost line by line if one assumes $f_{k}$ is a function rather than a number. Then one needs to replace $|f_{k}|$ in \eqref{eq: bourgainsetting} and \eqref{eq: L2sum} by $\|f_{k}\|_{L_{x}^{\infty}}$.

\medskip
We will provide more details about the precise form of the $H_x^1(\RRR^2)$ blowup data in Section \ref{sec:in_data}, and we will state a more detailed version of Theorem \ref{thm: main} in Theorem \ref{thm: mainrigor} below.

\begin{rem}
While our techniques are probabilistic, in light of the previous discussion on the randomized initial data, our main theorem provides a construction of $L^2_x(\RRR^2)$ log-log blowup solutions for the mass critical nonlinear Schr\"odinger equation which do not lie in $H^{s}_x(\RRR^2)$ for any $s>0$. To the best our knowledge, such examples were not previously known in the literature. 
\end{rem}
\begin{rem}
 We emphasize that unlike many random data results, \eqref{eq:nls} is locally well posed at $L^{2}_x(\RRR^2)$, which is the regularity at which we aim to construct our solutions. Consequently, we do not use the randomness to overcome ill-posedness for low regularity data, rather, we use randomization to construct a rough but highly structured perturbation of the original log-log dynamics.
\end{rem}

\subsection{Comparison with previous results}\label{ssec:compare}

\subsubsection{Log-log blowup in $H^{1}$}
We start with a quick review of the works of Merle and Rapha\"el, \cite{merle2005blow, merle2003sharp, merle2004universality, merle2006sharp}. Let us will focus on $H^{1}(\RRR^2)$ solutions $u$ to \eqref{eq:nls} with negative energy, zero momentum\footnote{One can always perform a Galilean transformation to set the momentum to zero, which does not change the mass and does not increase the energy.}, and with mass slightly above that of the ground state $Q$, see \eqref{eq: slight}. Via a variational argument and modulation theory, one can establish a geometric decomposition for the solution, given by
\begin{equation}\label{eq: mrg}
u(t,x)=\frac{1}{\lambda(t)}(\qbb+\epsilon)\left(\frac{x-x(t)}{\lambda(t)}\right)e^{i\gamma(t)},
\end{equation}
where $\qbb$ is a certain elliptic object which is a modification of $Q$, see \eqref{eq:qb}, such that certain orthogonality conditions given in \eqref{eq: modoth1}-\eqref{eq: modoth4} below, hold, and so that $b$ and $\epsilon$ are a priori small.

One may then reduce the study of \eqref{eq:nls} to that of the evolution\footnote{Heuristically, you now have 5 unknowns, $\epsilon, b(t),\lambda(t), x(t), \gamma(t)$ and you have five equations, \eqref{eq:nls}, \eqref{eq: modoth1}-\eqref{eq: modoth4}. Thus, one may expect the system is well determined.}of $\epsilon(t,x)$ and the parameters $b(t), \lambda(t), \gamma(t), x(t)$. It turns out that one should study this system in a rescaled time variable $s$ rather then the original time variable, where 
\begin{equation}\label{eq: timerescale}
\frac{dt}{ds}=\lambda^{2}.
\end{equation}
We note that $\lambda$ dictates the blowup rate, and the parameter $b$ dictates the evolution\footnote{Or more precisely, $\qbb$ is constructed in a such a way so that $b\sim-\frac{\lambda_{s}}{\lambda}$} of $\lambda$ in the sense that $b\sim -\frac{\lambda_{s}}{\lambda}$

A key estimate in the analysis of Merle and Rapha\"el is a local virial estimate\footnote{Note that this estimate only involves local $L^{2}$ information.},
\begin{equation}\label{eq: lv}
b_{s}\geq H(\epsilon)-2\lambda^{2}E-\Gamma_{b}^{1-C\eta}+o(1)(\int |\nabla\epsilon|^{2}+\epsilon^{2}e^{-|y|} ),
\end{equation}
where $\Gamma_b$ is a certain quantity which we define in \eqref{eq: gb} below, that satisfies
 \begin{equation}
e^{-(1+C\eta)\frac{\pi}{b}}\leq \Gamma_{b}\leq e^{-\frac{\pi}{b}(1-C\eta)}.
 \end{equation}
 for $C \eta \ll 1$. One hopes to derive from \eqref{eq: lv} that
 \begin{equation}\label{eq: upper}
b_{s}\geq -\Gamma_{b}^{1-C\eta}.
\end{equation}
Note that formula \eqref{eq: upper} is closely related to the sharp upper bound of the log-log blowup.  The main point is $H$ in \eqref{eq: lv} is some quadratic form, which will be coercive, dominating 
\[
\int |\nabla\epsilon|^{2}+\epsilon^{2}e^{-|y|} 
\]
up to six `bad' directions. Four bad directions will be handled via orthogonality of the modulation parameters\footnote{In practice, some extra is cancellation is needed, since one of the orthogonality conditions for the modulation parameters does not directly compensate for one of the bad directions associated with $H$ }, the other two are handled by energy and momentum conservation. We also remark here that when $E$ is negative, the term $-2\lambda^{2}E$ in \eqref{eq: lv} will not pose any problems for the analysis. In other words, heuristically, one only needs control on the positive part of $E$.

\begin{rem}
The estimate \eqref{eq: upper} already implies the sharp upper bound on the blowup rate for the log-log dynamics, which only differs from the direct scaling lower bound up to a double logarithm. To derive the sharp lower bound, one needs to introduce a truncated object $\zbb$, defined in \eqref{eq:zbb}, to further sharpen the analysis, see Section \ref{sec:ell} and Section \ref{sec: boot} below. 

We note that \eqref{eq: upper} is enough to drive the dynamics into a regime where 
\begin{equation}\label{eq: sosmall}
\lambda\ll e^{-e^{\Gamma_{b}^{-c}}}.
\end{equation}
 In this regime, the crucial observation in \cite{colliander2009rough} is that when $\lambda$ is small compared to $b$, one does not need the negative energy assumption anymore since $\lambda^{2}|E|$ can be treated as a small perturbation. A similar mechanism can also be applied to momentum, i.e. one does not require the strict zero momentum condition. 
\end{rem}

\subsection{$H^{s}$ stability of log-log blowup}
The study of the log-log dynamics in $H^s_x(\RRR^2)$ can be split into two stages. The first stage establishes rigidity of the dynamics in the sense that the solution will be driven towards some special, well-prepared initial data with an almost explicit form. The second stage establishes that for such well prepared data, its evolution can be understood via a bootstrap argument, see \cite{planchon2007existence}. Though both stages will rely on the same crucial ingredients from the analysis of Merle and Rapha\"el, the dynamics in the second stage is better understood since one can argue explicitly by bootstrap. 

The Cauchy problem \eqref{eq:nls} is locally well posed in $H^{s}_x(\RRR^2)$, for any $s>0$. Thus, to prove $H^{s}_x(\RRR^2)$ stability of log-log blowup dynamic is equivalent to proving that for those well-prepared initial data whose evolution can be characterized by the bootstrap estimates, the evolution is stable under $H^{s}_x(\RRR^2)$ perturbations. This fact is established in the work of Colliander and Rapha\"el \cite{colliander2009rough}. One crucial observation and heuristic is that since the solution $u$ is of the form
\[
u=\frac{1}{\lambda(t)}h(\lambda(t)), \qquad \|h(t)\|_{H^{1}}\sim 1,
\]
in the $H^{1}_x(\RRR^2)$ case, one may expect that in the $H^{s}_x(\RRR^2)$ case, the solution has a similar structure, with ``quantitative energy bounds''\footnote{we record this to give some intuition, but in practice, one needs to perform a frequency truncation to discuss the energy of $u$.}
\[
E(u)\sim \frac{1}{\lambda^{2-2s}}.
\]
Recall in \eqref{eq: lv}, the term $E(u)$ has been multiplied by $\lambda^{2}$ in the analysis, and will be formally of size $\lambda^{2s}$.   When $\lambda$ satisfies \eqref{eq:  sosmall}, this term can be treated perturbatively provided $s>0$. This also explains why $s=0$ is conceptually different from the case $s>0$.

\medskip
Unsurprisingly, a main challenge in the analysis of  \cite{colliander2009rough} is that since $u$ is not in $H^{1}(\RRR^2)$ anymore, the energy $E(u)$ is not well-defined (indeed, otherwise, it would be bound by a constant). To overcome this difficulty, one employs the I-method, introduced by Colliander, Keel, Staffilani, Takaoka and Tao \cite{colliander2002almost}, a ubiquitous method in the study of dispersive PDE which exploits energy conservation for low regularity data, and which is philosophically similar (although practically not completely equivalent) to the high-low method of Bourgain, \cite{bourgain1998refinements}. We note that it may be surprising that one can apply I-method for all $s>0$, whereas typically, such computations only work for certain $s>s_{0}$. Broadly speaking, this is because  one has a good \emph{a priori} understanding of the log-log asymptotics.

To briefly sketch the strategy of \cite{colliander2009rough}, one still considers the ansatz
\begin{equation}
u(t,x)=\frac{1}{\lambda(t)}(\qbb+\epsilon)\biggl(\frac{x-x(t)}{\lambda(t)}\biggr)e^{-i\gamma(t)}.
\end{equation}
 One applies the time-dependent operator $I_{N(t)}$ which truncates the high frequency part of the solution above $N(t)=\lambda(t)^{-(1+)}$. One then aims to study the evolution of $I_{N(t)}u(t)$, and to prove that the \emph{positive part} of energy $E(I_{N(t)}u(t))$ is controlled by $\lambda^{-2+2s}$ (formally speaking) via the I-method. In particular, one must establish that this energy can not be too large and positive, although it can be very negative\footnote{It is also emphasized in \cite{colliander2009rough} that the negative part of the energy always drives the solution to blowup.}. A key observation in \cite{colliander2009rough} is that the I-method is compatible with log-log bootstrap scheme, see also \cite{planchon2007existence}.

\subsection{Random $L^{2}$ perturbations}
The approach in \cite{colliander2009rough} breaks down for general $L_{x}^{2}$ perturbations. In this article, we will use randomized $L_{x}^{2}(\RRR^2)$ data $f^{\omega}$, defined in \eqref{eq:rand}, to perturb essentially the same well-prepared data as in \cite{colliander2009rough}. For the solutions $u(t,x)$ to \eqref{eq:nls}, we use the ansatz
\begin{equation}
u(t,x)=a(t,x)+F(t,x), 
\end{equation}
where 
\[
a=\frac{1}{\lambda(t)}(\qbb+\epsilon)\biggl(\frac{x-x(t)}{\lambda(t)}\biggr)e^{-i\gamma(t)} \qquad \textup{and} \qquad F=e^{it\Delta}f^{\omega}.
\]
Our $a$ will behave as the full solution $u$ of \cite{colliander2009rough}.

\medskip
The study of dispersive PDEs via a probabilistic approach was initiated by Bourgain~\cite{B94, B96} for the periodic nonlinear Schr\"odinger equation in one and two space dimensions, building upon the constructions of invariant measures by Glimm and Jaffe~\cite{GlimmJaffe} and Lebowitz, Rose and Speer~\cite{LRS}. Such questions were further explored by Burq and Tzvetkov~\cite{BT1, BT2} in the context of the cubic nonlinear wave equation on a three-dimensional compact Riemannian manifold. There has since been a vast body of research where probabilistic tools are used to study many nonlinear dispersive or hyperbolic equations at supercritical regularities, see for instance the works \cite{CO, Bourgain_Bulut1, BOP2, NS, Bringmann_dnlw, DNY19} as well as  references and discussion therein.

Certain global-in-time random data results in the compact setting which rely on invariant measures work equivalently in the focusing and defocusing cases, see \cite{B94}. However, in the absence of an invariant measure, the vast majority of existing large data probabilistic results treat only the \textit{defocusing} nonlinear Schr\"odinger and wave equations, see for instance \cite{BT4,Pocovnicu,  LM1, LM2, DLuM1, OP, Bringmann_scattering1, KMV} and references therein. We note that analogously to the deterministic theory, one may occasionally obtain ``small data'' type probabilistic results in the focusing setting, see for instance \cite{LM1}, although these are consequences of the local theory and do not relate to the large data probabilistic techniques.

There are two recent works in particular which treat the focusing problem with random initial data, outside the small data or local in time regimes. The first is work of Kenig and Mendelson \cite{KenigMendelson19} which studies the probabilistic stability of the soliton for the energy critical nonlinear wave equation on $\RRR^3$. In that work, the authors produce with high probability a family of radial perturbations of the soliton which give rise to global forward-in-time solutions of the focusing nonlinear wave equation that scatter after subtracting a dynamically modulated soliton. The proof relies on a new randomization procedure using distorted Fourier projections associated to the linearized operator around a fixed soliton. Another work, also in the context of nonlinear wave equations, is recent work of Bringmann \cite{Bringmann_blowup} on the probabilistic stability of blowup for the ODE blowup for the cubic nonlinear wave equation on $\mathbb{R}^3$. The proof in this latter paper relies on probabilistic Strichartz estimates in similarity coordinates, and in particular, does not require a randomization adapted to the blowup solution.

\medskip
Like in \cite{KenigMendelson19} and \cite{Bringmann_blowup}, in the current work we are in the large data yet perturbative regime. We note, however, that the geometric blowup we treat is quite distinct from the ODE blowup handled in \cite{Bringmann_blowup}. We leverage random data techniques to establish a bootstrap result,  stated in Lemma \ref{lem: boot} below. In the present work, compared to previous random data works, we do not need to use probabilistic improvements to overcome issues with deterministic wellposedness. Indeed, as mentioned previously, \eqref{eq:nls} is deterministically locally well-posed in $L^2_x(\RRR^2)$ via Strichartz estimates. However, we leverage the random data in two novel ways: first, we use it to obtain precise quantitative control on the wellposedness estimates for the bootstrap scheme, and second, we use the randomness to achieve the endpoint estimates\footnote{More precisely, though the I-method computation evolves terms with end point regularity, we apply the probalistic techiniques to prove they will behave like non-end points elements, thus the problem is still sub-critial rather than crtical. It should be noted both I-method and random data analysis are of sub-crtical nature, though in those problems, the original critical regularity may not be crtical in the usual sense.} for the I-method computation mentioned above. The former estimates are achieved in the spirit of the work of Bourgain \cite{Bourgain}, adapted to $\RRR^2$. One difference between our work and Bourgain's is that in light of the fact that our initial data lies at $_{x}^{2}(\RRR^2)$ regularity, there is no need to ``Wick-order'' the nonlinearity\footnote{The precise form of such an ordering is not as obvious in the Euclidean setting, however a continuous version of Wick-ordering is indeed possible.}. 

In the I-method computation, we exploit improved probabilistic estimates for the free evolution of the random initial data in a novel manner, mainly we use the fact that they can be made uniformly small in time to close the I-method estimates at the endpoint. Heuristically, one can view the free evolution of the random data as being not only equidistributed in space for a fixed time, but also roughly equidistributed in time. Hence, from the point of view of the time-scales associated to log-log blowup, the free evolution of the random data can be thought of as a source term which makes increasingly small contributions to the dynamics, and thus preserves the blowup mechanism.  Finally, we remark, that although a priori $a$ and $F$ are both $\mathcal{O}(1)$ in $L_{x}^{2}(\RRR^2)$, since the randomized initial data is non-concentrated, the $L^{2}$-pairing of the singular part, $a$, with the free evolution of the random data, $F$, will give a power of $\lambda$, and hence all such terms will be of perturbative nature in the modulation argument since $\lambda$ is so small. We will carry on a more thorough discussion of the proof in Section~\ref{sub: ov}.

\subsection{An overview for the proof of bootstrap lemma}\label{sub: ov}
We conclude the introduction by presenting an overview for the proof of the main bootstrap result, stated in Lemma \ref{lem: boot}. Two ingredients are required to close the bootstrap:
\begin{itemize}
\item Under the bootstrap assumption, the dynamics can be viewed as a perturbation of the log-log blowup dynamics.
\item Log-log blowup dynamics can upgrade the bootstrap assumptions to bootstrap estimates.
\end{itemize}

We will focus on the first ingredient, since the second part essentially follows from earlier works, in particular \cite{colliander2009rough}, building on the earlier works \cite{merle2006sharp, planchon2007existence}. There are three main factors which ensure the dynamics can be viewed as a perturbation of the log-log dynamics:
\begin{itemize}
\item According to the bootstrap assumptions, we have $\lambda\ll b$, and essentially all terms of form $\lambda^{\sigma}$, for $\sigma>0$ maybe treated as a perturbation. In particular, if one pairs the linear evolution of the randomized data $F$ with terms of form $\frac{1}{\lambda}h(x/\lambda)$ such that $h$ is somehow localized, one obtains a perturbative term.
\item The bootstrap assumption $t_{k+1}-t_{k}\lesssim k\lambda(t_{k})^{2}$ gives good control on how many local wellposedness (LWP) intervals we have throughout the analysis, and in particular, in every LWP interval $[a,b]$ such that $\lambda(t)\sim 2^{-k}$ and $|b-a|\sim \lambda^{2}(t_{k})$, we can establish probabilistic wellposedness, based on bootstrap assumption \eqref{eq: bah1control}.
\item Finally, based on the the probabilistic wellposedness in every LWP interval, one can perform an I-method type energy estimate combined with random data estimates to obtain good control on $E(I_{N(t)}u)$ (or more precisely speaking, the positive part of this quantity), which will ensure the log-log dynamics persist. One also needs to control $P(I_{N}u)$, but this is relatively easier.
\end{itemize}

Key ingredients in the current article are the development of suitable probabilistic wellposedness in every LWP interval, and the derivation of good energy estimates for $E(I_{N}u)$. One may compare these ingredients to those in \cite{colliander2009rough}, in which the usual $H_x^{s}(\RRR^2)$ wellposedness is used in every LWP interval, and a more classical version of I-method is applied. As noted earlier, although there are certainly crucial differences between the current work and \cite{colliander2009rough}, we fundamentally rely on the observation from that work that the I-method is compatible with the log-log bootstrap regime\footnote{In some sense, all I-method arguments rely on good control on the numbers of LWP intervals, and a good understanding of wellposedness estimates within every LWP interval}.

\subsection{Organization of paper}
In Section \ref{sec:prelim}, we introduce some probabilistic and deterministic preliminaries. In Section \ref{sec:in_data} we describe the initial data and introduce the bootstrap assumptions. We will elaborate on the probabilistic estimates and I-method type computation in Section \ref{sec:pi}. In Section \ref{sec: sig}, we will provide a relatively detailed sketch about how such energy estimate plus the log-log dynamic close the bootstrap scheme in Section \ref{sec: boot}.

\subsection{Notation}
We use $\Lambda:=1+y\cdot \nabla$ to denote the generator of the $L^{2}_x(\RRR^2)$ scaling. When we write for some $f$ one has $f=f_{1}+if_{2}$, implicitly, we mean $f_{1}:=\Re f$ and $f_{2}:=\Im f$. We denote by $C > 0$ an absolute constant which only depends on fixed parameters and whose value may change from line to line. We write $X \lesssim Y$ to indicate that $X \leq C Y$ and we use the notation $X \sim Y$ if $X \lesssim Y \lesssim X$. Moreover, we write $X \lesssim_\nu Y$ to indicate that the implicit constant depends on a parameter $\nu$ and we write $X \ll Y$ if the implicit constant should be regarded as small. We will write $c+$ to denote $c+ \varepsilon$ for an arbitrary $\varepsilon > 0$, and similarly for $c-$. We also use the notation $\langle x \rangle := (1+x^2)^{1/2}$.

\subsection{Acknowledgements} The authors thank Carlos Kenig and Gigliola Staffilani for interesting and helpful discussions.

\section{Preliminaries}\label{sec:prelim}
\subsection{I-operator}\label{sub: imethod}
Following \cite{colliander2002almost, colliander2009rough}, let $0 < s < 1$ and let $m: \mathbb{R}^+ \to \mathbb{R}^+$ be a smooth, monotone function which satisfies $m(|\xi|) = 1$ for $0 \leq |\xi| \leq 1$, and $m(|\xi|) = |\xi|^{s-1}$ for $|\xi| \geq 2$. Let $N \gg 1$ and define
\[
m_N(\xi) = m\left(\frac{|\xi|}{N} \right),
\]
and note that
\begin{equation}
m_N(|\xi|) =\begin{cases}
1 & |\xi| < N \\
\left( \frac{N}{|\xi|}\right)^{1-s} & |\xi| > 2N.
\end{cases}
\end{equation}
The operator $I_N$ is the Fourier multiplier associated to $m_N$:
\[
\widehat{I_N f}(\xi) = m_N(\xi) \widehat{f}(\xi)
\]
 and we note that
 \[
\|f\|_{H^s} \lesssim \|I_N \langle D \rangle f \|_{L^2} \lesssim N^{1-s} \|f\|_{H^s}.
 \]
 \begin{rem}
The operator $I_N$ is also strong-type $(p,p)$ for all $1\leq p\leq \infty$, uniformly in $N$.
\end{rem}

\subsection{Strichartz Estimates}
We recall the classical Strichartz estimates, which play a important role in the local theory of NLS.
\begin{definition}[Admissible pairs]
For $d \geq 1$ we say a pair of exponents $(q,r)$ is Schr\"odinger admissible if
\begin{align}
\frac{2}{q} + \frac{d}{r} = \frac{d}{2}, \qquad 2 \leq q,r \leq \infty, \quad \textup{and} \quad (d,q,r) \neq (2,2,\infty).
\end{align}
For a fixed spacetime slab, $I \times \mathbb{R}^d$, we define the Strichartz norm
\begin{align}
\|u\|_{S(I)} := \sup_{(q,r) \textup{ admissible}} \|u\|_{L^q_t L^r_x (I \times \RRR^d)}.
\end{align}
We let $S(I)$ denote the closure of all test functions under this norm, and let $N(I)$ denote its dual.
\end{definition}
\begin{rem}
In dimension $d= 2$, the supremum must actually be restricted to a closed subset to avoid the inadmissible endpoint.
\end{rem}
\begin{prop}[Strichartz estimates, cf. \cite{keel1998endpoint, tao2006nonlinear, cazenave2003semilinear}]
Let $0 \leq s \leq 1$, let $I$ be a compact time interval, and let $u: I \times \RRR^d \to \mathbb{C}$ be a solution to the forced Schr\"odinger equation
\[
iu_t + \Delta u = F.
\]
Then for any $t_0 \in I$, we have
\[
\||\nabla|^s u \|_{S(I)} \lesssim \|u(t_0) \|_{\dot H^s_x} + \||\nabla|^s F \|_{N(I)}.
\]
\end{prop}
\begin{prop}[Bilinear Strichartz estimates, cf. \cite{bourgain1998refinements}]\label{prop:bilin}
Let $f_{1}, f_{2}$ be two $L_{x}^{2}(\RRR^2)$ functions, and let $N\geq M$. Then one has 
\begin{equation}
\|e^{it\Delta}P_{N}f_{1}e^{it\Delta}P_{M}f_{2}\|_{L_{t,x}^{2}}\lesssim \biggl(\frac{M}{N}\biggr)^{1/2}\|f_{1}\|_{L_{x}^{2}}\|f_{2}\|_{L_{x}^{2}}.
\end{equation}
\end{prop}

We now turn to the definition of $X^{s,b}$ spaces, \cite{bourgain1993fourier}:

\begin{definition}
The space $X^{s,b}(\RRR \times \RRR^d)$ is the closure of test functions under the norm
\[
\|u\|_{X^{s,b}(\RRR \times \RRR^d)} := \| \langle \xi \rangle^s \langle \tau - |\xi|^2 \rangle \widehat{u}(\xi ,\tau) \|_{L^2_{\xi, \tau}}.
\]
\end{definition}
Recall that $X^{s,b}$ embeds into $C_t^0 H^s_x$ for $b > \frac{1}{2}$. The restricted version of the space on $[-\delta, \delta] \times \RRR^d$ is defined by
\[
\|u\|_{X^{s,b, \delta}} :=  \inf \{\|\widetilde{u}\|_{X^{s,b}(\RRR \times \RRR^d)} \,:\, \widetilde{u} \big|_{[-\delta, \delta]} = u \}.
\]
We recall that free solutions lie in $X^{s,b}$ \emph{locally} in time but not globally. An important property of $X^{s,b}$ spaces is the following:

\begin{lem}
Let $Y$ be a Banach space of functions on $\mathbb{R} \times \RRR^d$ with the property that
\[
\| e^{it \tau_0} e^{ it \Delta} f \|_Y \lesssim \|f\|_{H^s_x}
\]
for all $f \in H^s$ and $\tau_0 \in \RRR$. Then we have
\[
\|u\|_Y \lesssim_b \|u\|_{X^{s,b}(\RRR \times \RRR^d)}.
\]
\end{lem}

We will also need a multilinear transfer principle for $X^{s,b}$ spaces:
\begin{prop}[Transfer Principle, cf. \cite{HHK}]
Let $b > \frac{1}{2}$, $Y = L^q_t L^r_x$ for $1 \leq p,q \leq \infty$ and $\mathcal{T}$ a $k$-linear operator such that
\[
\| \mathcal{T}(e^{it\Delta} f_1, \ldots, e^{it\Delta} f_k) \|_Y \lesssim \prod_{j=1}^k \|f_j \|_{H^s_j},
\]
then
\[
\| \mathcal{T}(u_1, \ldots,u_k) \|_Y \lesssim \prod_{j=1}^k \|u_j \|_{X^{s_j, b}}.
\]
\end{prop}
We will use the transfer principle repeatedly throughout our estimates in order to combine Strichartz estimates with $X^{s,b}$ spaces.

\subsection{Random data preliminaries}
Here we collect some of the random data results which we will use in the sequel. Recall below $\FF=e^{it \Delta} f^\omega$, where $f^\omega$ has been defined in \eqref{eq:rand}. We begin with the following $\ell^\infty$ Gaussian bound:

\begin{lem}\label{lem:l_inf_gaussian}
For every $\varepsilon > 0$, there exists $C, c> 0$ such that
\begin{align} \label{ell_infty}
	\mathbb{P}(\{  \langle n \rangle^{-\varepsilon}|g_n(\omega)| > K \}) \leq Ce^{-c K^2}
\end{align}
\end{lem}

Next we record a standard probabilistic estimate.
\begin{lem}[\protect{\cite[Lemma 3.1]{BT1}}] \label{lem:large_dev}
 Let $\{g_n\}_{n=1}^{\infty}$ be a sequence of complex-valued independent identically distributed (iid) mean-zero Gaussian random variables on a probability space $(\Omega, {\mathcal A}, \mathbb{P})$. Then there exists $C > 0$ such that for every $p \geq 2$ and every $\{c_n\}_{n=1}^{\infty} \in \ell^2(\mathbb{N}; \mathbb{C})$, we have
 \begin{equation*}
  \Bigl\| \sum_{n=1}^{\infty} c_n g_n(\omega) \Bigr\|_{L^\rho_\omega} \leq C \sqrt{\rho} \Bigl( \sum_{n=1}^{\infty} |c_n|^2 \Bigr)^{1/2}.
 \end{equation*}
\end{lem}

We will also use the following variant of \cite[Lemma 4.5]{Tz10} to bound the probability of certain subsets of the probability space.
\begin{lem}\label{lem:large_dev2}
Let $F$ be a real valued measurable function on a probability space $(\Omega, \mathcal{A}, \mathbb{P})$. Suppose that there exists $\alpha > 0$, $N > 0$, $k\in \mathbb{N} \setminus \{0\}$ and $C > 0$ such that for every $\rho \geq \rho_0$ one has
\begin{align}
\|F\|_{L^\rho_\omega} \leq C N^{-\alpha} \rho^{\frac{k}{2}}.
\end{align}
Then, there exists $C_1$, and $\delta$ depending on $C$ and $\rho_0$ such that for $K > 0$
\begin{align}
\mathbb{P}( \omega \in \Omega : |F(\omega)| > K) \leq C_1 e^{-\delta N^{\frac{2\alpha}{k}} K^{\frac{2}{k}}}.
\end{align}
\end{lem}

\begin{lem}
Let $1 \leq p, q < \infty$, then for all $\rho \geq \max(p, q)$ we have
\begin{equation}\label{eq: lp0}
\|\FF\|_{L_{\omega}^{\rho}L_{t}^{q} L^p_x ([0,1]\times \mathbb{R}^{2})}\lesssim_\rho \|f\|_{L_{x}^{2}}.
\end{equation}
In particular, there exists $C ,c > 0$ such that
\begin{align}
\mathbb{P}( \|\FF\|_{L_{t}^{q} L^p_x ([0,1]\times \mathbb{R}^{2})} > K) \leq C e^{-c K^2 / \|f\|_{L_{x}^{2}}^2}.
\end{align}
\end{lem}
\begin{proof}
We use Minkowski's inequality and Lemma \ref{lem:large_dev} to estimate
\begin{align}
\|\FF\|_{L_{\omega}^{\rho}L_{t}^{q} L^p_x ([0,1]\times \mathbb{R}^{2})} &\leq \|\FF\|_{L_{t}^{q} L^p_x ([0,1]\times \mathbb{R}^{2})L_{\omega}^{\rho}} \\
& \lesssim \left( \sum_k \|e^{it\Delta} P_k f\|_{L_{t}^{q} L^p_x ([0,1]\times \mathbb{R}^{2})} \right)^{1/2}.
\end{align}
The result then follows from  H\"older's inequality and the unit-scale Bernstein inequality, while the estimate on the probability follows from Lemma \ref{lem:large_dev2}.
\end{proof}

\begin{rem}
Essentially repeating the proof of the previous lemma, one may apply Minkowski and using the $\ell^{2}$ summability \eqref{eq: L2sum}, to improve \eqref{eq: lp0} into 
\begin{equation}\label{eq: lp00}
\bigl\|(\sum_{N}\|P_{N}\FF\|_{L_{t,x}^{p}([0,1] \times \RRR^2)}^{2})^{1/2}\bigr\|_{L_{\omega}^{\rho}}\lesssim_\rho 1.
\end{equation}
where $N\geq 1$ ranges over all dyadic integers. Via interpolation with the $L_{t}^{\infty}L_{x}^{2}$ bound and H\"older's inequality in time, up to an exceptional set of exponentially small probability, one has for $2\leq q\leq p < \infty$ that
\begin{equation}\label{Lp}
\|F\|_{L_{t}^{p}L_{x}^{q}([0,1]\times \mathbb{R}^{2})}\lesssim 1, 
\end{equation}
and 
\begin{equation}\label{Lp2}
\sum_{N}\|P_{N}\FF\|_{L_{t}^{p}L_{x}^{q}([0,1]\times \mathbb{R}^{2})}^{2}\lesssim 1.
\end{equation}
\end{rem}

Finally, we will need a multilinear Gaussian estimate. We state a slightly simplified version of this estimate compared to the reference since this will suffice for our purposes.

\begin{lem}[\protect{Cf. \cite[Proposition 2.4]{TTz}}]\label{lem:multi_gauss}
Let $\{g_n\}$ be iid mean-zero Gaussian random variables, and let 
\[
* = \{n_1, n_2, n_3\in \mathbb{Z}^{2} \,:\, n_2 \neq n_1, n_3\}.
\]
Consider
\begin{align}
	&G(\omega) =  \sum_{*}c(n_1, n_2, n_3) g_{n_1}(\omega) \overline{g}_{n_2}(\omega) g_{n_3}(\omega) 
\end{align}
where $c(n_1, n_2, n_3)$ are complex numbers. Then there exits $C, c > 0$ such that
\begin{equation}\label{eq: mgem1}
\mathbb{P}\{ |G| >  K \|G\|_{L^2_\omega} \} \leq C e^{ - c K^2}.
\end{equation}
\end{lem}

\subsection{Elliptic objects}\label{sec:ell}
We recall, as detailed in the introduction, that analysis of Merle and Rapha\"el in  \cite{merle2005blow} for log-log blowup solutions begins with a geometric decomposition of the solution $u$, given by
\begin{equation}
u(t,x)=\frac{1}{\lambda(t)}(\qbb+\epsilon)\left(\frac{x-x(t)}{\lambda(t)}\right)e^{i\gamma(t)},
\end{equation}
where $\qbb$ is a certain elliptic object which is a modification of $Q$ and where $\epsilon$ is a priori small in $H^1_x(\RRR^2)$. Schematically, this implies a scaled and translated version of the ground state $Q$ is a good approximation for $u$, up to modulation. However, in practice, one considers a modification of $Q$ to capture the sharp log-log blowup dynamics \cite{merle2003sharp, merle2004universality, merle2006sharp}. This modification of $Q$ relies on certain elliptic objects, $Q_{b}, \qbb, \zeta_{b}$ and $\zbb$, which we describe in this subsection. We will not list all the properties of these objects, which are indeed crucial to log-log analysis but will not explicitly be used in this article since we rely on previous results which establish the existence of such solutions in the energy space. Instead we focus only on the properties which are most relevant to the current work and refer to \cite{merle2003sharp, merle2004universality, merle2006sharp} for more details. One may refer to \cite[Proposition 1 and Lemma 2]{merle2006sharp} for further details.

\medskip
Throughout this subsection, $b$ and $\eta$ will be used to denote small positive numbers, $C$ will denote a universal constant, and one should have in mind that $C\eta\ll 1$. We let 
\[
R_{b}:=\frac{2}{|b|}\sqrt{1-\eta}.
\]
Let $\qb$ be a modification of $Q$ which solves
\begin{equation}\label{eq:qbb}
\begin{cases}
\Delta \qb-\qb+ib\Lambda \qb+|\qb|^{2}\qb=0,\\
 Q_{b}e^{i\frac{b|y|^{2}}{4}}>0 \text{ in } B_{R_{b}},\\
 Q_{b}(R_{b})=0.
\end{cases}
\end{equation} 
Now, let  $R_{b}^{-}:=\sqrt{1-\eta}R_{b}\sim b^{-1}$ be a constant sightly smaller than $R_{b}$, and let $\phi_{b}(x)$ be a smooth cut-off function with $\phi_b(x) \equiv 1$ on  $|x|\leq R_{b}^{-}$ and  $\phi_b(x) \equiv 0$ for $|x|\geq R_{b}$. We define $\qbb$ be the cut-off version of $\qb$, namely $\qbb = \phi_{b} \qb$, and we let
\begin{equation}\label{eq:qb}
\Delta \qbb-\qbb+ib\Lambda \qbb+|\qbb|^{2}\qbb:=-\Psi_{b}.
\end{equation}
Note that $\qbb(x)$ decays exponentially as $|x|\rightarrow \infty$, thus asymptotically the nonlinearity of \eqref{eq:qb} vanishes. 
 
 Following the work of Merle and Rapha\"el \cite{merle2006sharp}, one introduces the tail $\zb$, which is the unique radial solution to
 \begin{equation}
 \begin{cases}
 \Delta \zeta_{b}-\zeta_{b}+ib\Lambda \zeta_{b}+ib\zeta_{b}=\Psi_{b},\\
 \zeta_{b}\in \dot{H}^{1}_x(\RRR^2).
 \end{cases}
 \end{equation}
 It turns out that $\zeta_{b}$ just misses $L_{x}^{2}(\RRR^2)$, or more precisely, if we define
 \begin{equation}\label{eq: gb}
 \Gamma_{b}:=\lim_{|y|\rightarrow \infty}y^{2}|\zeta_{b}|^{2},
 \end{equation}
then this limit exists and we have
 \begin{equation}
e^{-(1+C\eta)\frac{\pi}{b}}\leq \Gamma_{b}\leq e^{-\frac{\pi}{b}(1-C\eta)}.
 \end{equation}
The quantity $\Gamma_{b}$ appears frequently in the log-log blowup analysis, and this scale plays a crucial role. A useful heuristic to keep in mind is that all terms of size $\Gamma_{b}^{1+}$ are acceptable. For example, if one modifies $\zb$ into $\zb'$ so that $\|\zbb-\zbb'\|_{H^{1}}\lesssim \Gamma_{b}^{1+}$, then, heuristically, there are no difference between those two terms in the log-log analysis.

To overcome the failure of $L^{2}_x(\RRR^2)$ integrability of $\zb$, we introduce a cut-off version of this object, denote by $\zb$, as follows:  let $\psi$ be a bump function localized at $|x|\lesssim 1$ and let $a$ be a small number. Let 
\[
A=A_{b}:=e^{\frac{\pi}{b}}, \quad \psi_{A}(x):=\psi(\frac{x}{A}),
\]
and let 
\begin{equation}\label{eq:zbb}
\zbb=\psi_{A}\zb.
\end{equation}
Note that $\Gamma_{b}^{-a/2}\leq A_{b}\leq \Gamma_{b}^{-3a/2}$ and 
\begin{equation}
\int |\zbb|^{2}\leq \Gamma_{b}^{1-C\eta}.
\end{equation}
One also records 
\begin{equation}
\Delta \zbb-\zbb+ib\Lambda \zbb:=\Psi_{b}+F_{b}.
\end{equation}
The crucial fact about the tails, used essentially \cite[(4.20)]{merle2006sharp}, is that
\begin{equation}\label{eq: flux}
-Re(\zbb, \Lambda F_{b})\geq c\Gamma_{b}.
\end{equation}

We conclude this section by listing some useful estimates for $\qbb$. Most of the time, however, it will be enough to think of it as a function which decays exponentially, uniformly in $b$.
\begin{enumerate}
\item One has
\begin{equation}
|E(\qbb)|\lesssim \Gamma_{b}^{1-C\eta}, \qquad P(\qbb)=0.
\end{equation}
\item $\qbb$ is uniformly close to $Q$, and
\begin{equation}\label{eq: bclose}
\|e^{(1-\eta)\frac{\theta(|b||y|)}{|b|}}(\qbb-Q)\|_{C^{3}}\xrightarrow{b\rightarrow 0} 0,
\end{equation}
where 
\begin{equation}
\theta(r)=\mathbf{1}_{\{0\leq r\leq 2\}}\int_{0}^{r}\sqrt{1-\frac{z^{2}}{4}}dz+1_{r>2}\frac{\theta(2)}{2}r,
\end{equation}
and $\theta(2)=\frac{\pi}{2}$.
\item One has the following non-degeneracy  with parameter $b$
\begin{equation}
\biggl\|e^{(1-\eta)\frac{\theta(|b||y|)}{|b|},}(\frac{\partial}{\partial_{b}}\qbb+i\frac{|y|^{2}}{4}Q)\biggr\|\xrightarrow{b\rightarrow 0} 0.
\end{equation}
\item $\qbb$ has strictly supercritical mass and 
\begin{equation}
\|\qbb\|^{2}-\|Q\|_{2}^{2}\sim b^{2}.
\end{equation}
\end{enumerate}

\section{Preparation of initial data and setting up the bootstrap}\label{sec:in_data}
In this section, we describe the necessary steps in order to set up the main bootstrap lemma and prove the main theorem.
\subsection{Description of initial data and statement of main results}\label{subsec: ini}
Recall that we consider a randomized $L^2_x(\RRR^2)$ function $f^{\omega}$, given by
\begin{equation}\label{eq: randomized}
f^{\omega}(x)=\sum_{k}\int f_{k}g_{k}(\omega)\psi_{k}(\xi)e^{ix\xi}d\xi
\end{equation}
where $\{g_{k}\}_{k\in \ZZZ^{2}}$ are iid mean-zero complex Gaussian random variables, and where $\psi_k$ and $f_k$ are defined in \eqref{eq:psi_k} and \eqref{eq:p_k}. We recall that we assume that the $f_k$ satisfy the decay condition
\begin{equation}
|f_{k}|\leq \frac{1}{|k|}, \quad k\neq 0
\end{equation}
and normalization
\begin{equation}
\sum_{k}|f_{k}|^{2}=1.
\end{equation}
We will use $F$ to denote the linear evolution of the random data $f^\omega$, that is 
\begin{equation}\label{eq: linearevo}
F(t,x)=F^{\omega}(t,x)=e^{it\Delta}f^{\omega}.
\end{equation}

We let $a_{0}$ be the well-prepared initial data, given by
\begin{align}\label{eq:a0}
a_{0}=\frac{1}{\lambda_{0}}(\qbo+\epsilon_{0})(\frac{x-x_{0}}{\lambda_{0}}),
\end{align}
and one may, without loss of generality, take $x_{0}=0$. We make the following assumptions which will ensure that we fall in the bootstrap regime of the log-log dynamics:
\begin{itemize}
\item Smallness of $b_{0}$
\begin{equation}\label{eq: boini}
0<b_{0}\ll 1,
\end{equation}
\item smallness of $\lambda_{0}$
\begin{equation}\label{eq: lamoini}
0<\lambda_{0}\leq e^{-e^{2\pi/3b_{0}}} ,
\end{equation}
\item smallness of extra mass
\begin{equation}\label{eq: smallmassini }
\|\eo\|_{L_{x}^{2}}\ll 1,
\end{equation}
\item $H^{1}$ smallness of $\eo$,
\begin{equation}\label{eq: smallnss}
\int |\nabla \epsilon_{0}|^{2}+|\epsilon_{0}|^{2}e^{-|y|}\leq \Gamma_{b_{0}}^{3/4},
\end{equation}
\item control of energy and momentum
\begin{equation}\label{eq: energyini}
\lambda_{0}^{1/2}|E(a_{0})|\leq 1,
\end{equation}
\begin{equation}\label{eq: momeini}
\lambda_{0}^{1/2}|P(a_{0})|\leq 1
\end{equation}
\end{itemize}
and the following four orthogonality conditions
\begin{equation}\label{eq: inioth1}
(\epsilon_{1,0}, |y|^{2}\Sigma_{b_{0}})+(\epsilon_{2,0}, |y|^{2}\Theta_{b_{0}})=0,
\end{equation}
\begin{equation}\label{eq: inioth2}
(\epsilon_{1,0}, y\Sigma_{b_{0}})+(\epsilon_{2,0}, y\Theta_{b_{0}})=0,
\end{equation}
\begin{equation}\label{eq: inioth3}
-(\epsilon_{1,0},\Lambda \Theta_{b_{0}})+(\epsilon_{2,0}, \Lambda \Sigma_{b_{0}})=0,
\end{equation}
\begin{equation}\label{eq: inioth4}
-(\epsilon_{1,0}, \Lambda^{2}\Theta_{b_{0}})+(\epsilon_{2,0}, \Lambda^{2}\Sigma_{b_{0}})=0,
\end{equation}
where
\begin{equation}
\epsilon_{0}=\epsilon_{1,0}+i\epsilon_{2,0}, \quad \qbo=\Sigma_{b_{0}}+i\Theta_{b_{0}}
\end{equation}
and $(f,g)$ denotes the real $L^2_x$ inner product. We remark that such initial data $a_{0}$ are easy to construct by the work of Merle and Rapha\"el \cite{merle2006sharp}. Indeed, one simply finds $H^{1}_x(\RRR^2)$ initial data, with non-positive energy and mass slightly about that of the ground state, and evolves it under the flow of \eqref{eq:nls} until it is close enough to the blowup time. 

\medskip
Here and in the sequel, we assume that $f$ and $a_0$ satisfy the above conditions. We are now prepared to state our main result. 
\begin{thm}\label{thm: mainrigor}
Fix $f$ satisfying the above conditions, then there exists a universal constant $\lambda_0^* > 0$ such that for all $0 < \lambda_{0} <\lambda_0^*$, there exists a subset $\Sigma \subset \Omega$, and constants $C, c > 0$ so that
\[
\mathbb{P}(\Sigma) \geq 1 - C e^{- 1/\lambda_0^c},
\]
such that for all  $\omega \in \Sigma$, there exists a solution solution $u(t,x)$ to \eqref{eq:nls} with initial data $u_0 = a_0 + f^\omega$ which will blowup in finite time $0 < T=T_{\omega} \ll 1$ according to log-log law in the following sense: there are two small, fixed positive numbers $s, \delta$ so that
\begin{equation}
u(t,x)=a(t,x)+F(t,x), \quad  a(t,x)=\frac{1}{\lambda(t)}(\qbb+\epsilon)\bigl(\frac{x-x(t)}{\lambda(t)}\bigr), \quad F = e^{it\Delta} f^\omega
\end{equation}
and 
\begin{equation}
\lambda(t)^{-1}\sim \frac{\sqrt{\ln |\ln |T-t|}|}{\sqrt{T-t}}, \quad \|\qbb+\epsilon\|_{H^{s}}\sim 1,
\end{equation}
and for some $N(t)=\lambda(t)^{-1-\delta}$, 
\begin{equation}
\int |\nabla I_{N(t)\lambda(t)} \epsilon|^{2}+|\epsilon|^{2}e^{-|y|}\xrightarrow{t\rightarrow T} 0.
\end{equation}
\end{thm}
One may refer to Section \ref{sub: imethod} for the definition and properties of the $I$-operator.
\begin{rem}
The large probability in the statement of the theorem can be understood in two ways. If one fixes $a_{0}$, and studies the evolution of $a_{0}+\aaa f^{\omega}$, then with probability $\geq 1-e^{-1/\alpha^c}$ the conclusion of Theorem \ref{thm: mainrigor} holds, provided $\aaa$ is sufficiently small. Alternatively, one may fix $f$, but consider $\lambda_{0}$ and $b_{0}$ sufficiently small, since the definition of $a_{0}$ is essentially be given by $\lambda_{0}$ and $b_{0}$. Then the conclusion of Theorem \ref{thm: mainrigor} holds with probability $\geq 1-e^{-1/\lambda_0^c}$. 
\end{rem}
\begin{rem}
At first glance, the statement of the main theorem may seem surprising, since one could choose $a_{0}$ so concentrated that one does not even need the $L_{x}^{2}(\RRR^2)$ smallness of $f^{\omega}$, or alternatively smallness of $\aaa$. One should still view the free evolution of the random data as a (small) perturbation (around a complicated object)  since the requirement that $a_{0}$ be concentrated, together with the fact that randomized functions are equidistributed in space still decouples these terms from one another, and thus the resulting interaction is still expected to be small.
\end{rem}
\begin{rem}
 Just as the $H^{1}_x(\RRR^2)$ case and the $H^{s}_x(\RRR^2)$ case for $s>0$, one can study the convergence of the concentration point $x(t)$ to establish that the blowup point is well defined, and one can prove (non)concentration properties of the radiation $\epsilon$ at the blowup point. We refer interested readers  to \cite{colliander2009rough} since these arguments follow in an identical manner in our setting.
\end{rem}

The dynamics described in the main theorem will be characterized by the bootstrap lemma in the next subsection.

\subsection{Bootstrap set-up}

Let $u$ be the solution to \eqref{eq:nls} with initial data $u_{0}$, we will use the ansatz 
\begin{equation}\label{eq: ans}
\begin{aligned}
u(t,x)&=a(t,x)+F(t,x),\\
a(t,x)&=\frac{1}{\lambda(t)}(\qbb+\epsilon)(\frac{x-x(t)}{\lambda(t)})e^{-i\gamma(t)},
\end{aligned}
\end{equation}
Via continuity of the flow in $L_{x}^{2}$ and our initial orthogonality condition \eqref{eq: inioth1} to \eqref{eq: inioth4}, we can ensure, at least locally in $t\in [0,T_{0}]$, $T_{0}$ small, one has orthogonality conditions
\begin{equation}\label{eq: modoth1}
(\epsilon_{1}, |y|^{2}\Sigma_{b})+(\epsilon_{2}, |y|^{2}\Theta_{b})=0,
\end{equation}
\begin{equation}\label{eq: modoth2}
(\epsilon_{1}, y\Sigma_{b})+(\epsilon_{2}, y\Theta_{b})=0,
\end{equation}
\begin{equation}\label{eq: modoth3}
-(\epsilon _{1}, \Lambda \Theta_{b})+(\epsilon_{2},\Lambda \Sigma_{b})=0,
\end{equation}
\begin{equation}\label{eq: modoth4}
-(\epsilon_{1}, \Lambda^{2}\Theta_{b})+(\epsilon_{2}, \Lambda^{2}\Sigma_{b})=0.
\end{equation}

and furthermore 
\begin{equation}
x(0)=x_{0}, \lambda(0)=\lambda_{0}, \gamma(0)=0.
\end{equation}

where we use notation $\epsilon=\epsilon_{1}+i\epsilon_{2}$ and $\qbb=\Sigma_{b}+i\Theta_{b}$.

We will focus on the evolution of $a$, and one has 
\begin{equation}\label{eq: eqfora}
i\partial_{t}a+\Delta a=-|a+F|^{2}(a+F)=-|a|^{2}a-(|a+F|^{2}(a+F)-|a|^{2}a).
\end{equation}

We will need to two parameter  $s$, $\delta$ in the rest of the article. 
We always assume 
\begin{equation}
0<\delta\ll s\ll1,
\end{equation}
and in particular for any small constant $c$ involved in our analysis, one has 
\begin{equation}
(1-cs)(1+\delta)<1.
\end{equation}

\medskip
Now we are ready to state the main bootstrap lemma. Let $u$ solve \eqref{eq:nls}  in $[0,T]$ with initial data $u_{0}$ described as in Subsection \ref{subsec: ini}, with ansatz \eqref{eq: ans} so that \eqref{eq: modoth1}-\eqref{eq: modoth4} holds.
Since  $\lambda(t)$ is essentially monotone decreasing, or more precisely by bootsrap assumption \eqref{eq: bamono},  we can divide $[0,T]$ into $\cup_{k=k_{0}}^{k_{1}=1}[t_{k}, t_{k+1}]$, such that$\lambda(0)\sim 2^{-k_{0}}$ and $\lambda(T)\sim 2^{-k_{1}}$, $k_{0}\leq k\leq k_{1}$, and $\lambda(t)\sim 2^{-k}t\in [t_{k}, t_{k+1}]$.
\begin{lem}[Bootstrap Lemma]\label{lem: boot}
Suppose that $u(t,x)$ solves \eqref{eq:nls} on $[0,T]$ and satisfies the following bootstrap assumptions for $t\in [0,T]$:
\begin{equation}\label{eq: basmall}
0<b(t), \quad \|\epsilon\|_{L^{2}}+b(t)<\alpha,
\end{equation}
\begin{equation}\label{eq: bamono}
\forall t\leq t'\in [0,T], \quad \lambda(t')\leq \frac{3}{2}\lambda(t)
\end{equation}
\begin{equation}\label{eq: balambdasmall}
\lambda(t)\leq e^{-\frac{1}{\Gamma_{b}}^{2/3}},
\end{equation}
\begin{equation}\label{eq: bah1control}
\int |I_{N(t)\lambda(t)}\epsilon|^{2}+|\epsilon|^{2}\leq \Gamma_{b}^{2/3},
\end{equation}
\begin{equation}\label{eq: balwpinteval}
t_{k+1}-t_{k}\lesssim k\lambda(t_{k})^{2}\sim k2^{-2k},
\end{equation}
then
\begin{equation}\label{eq: besmall}
0<b(t), \quad  \|\epsilon\|_{L^{2}}+b(t)<\frac{\alpha}{2},
\end{equation}
\begin{equation}\label{eq: bemono}
\forall t\leq t'\in [0,T], \quad \lambda(t')\leq \frac{5}{4}\lambda(t)
\end{equation}
\begin{equation}\label{eq: belambdasmall}
\lambda(t)\leq e^{-\frac{1}{\Gamma_{b}}^{3/4}},
\end{equation}
\begin{equation}\label{eq: beh1control}
\int |I_{N(t)\lambda(t)}\epsilon|^{2}+|\epsilon|^{2}\leq \Gamma_{b}^{3/4},
\end{equation}
\begin{equation}\label{eq: belwpinteval}
t_{k+1}-t_{k}\lesssim \sqrt{k}\lambda(t_{k})^{2}\sim \sqrt{k}2^{-2k},
\end{equation}
\end{lem}
\begin{rem}
Formally speaking, the asymptotic dynamics gives 
\[
b_{s}\sim -\Gamma_{b},
\]
\[
\int |I_{N(t)\lambda(t)}\epsilon|^{2}+|\epsilon|^{2}\lesssim \Gamma_{b},
\] 
\[
\lambda\sim e^{-\Gamma_{b}^{-1}}
\]
\[
t_{k+1}-t_{k}\sim \ln\ln k
\]
 and the mass conservation law gives 
 \[
 \|\epsilon\|_{2}+b^{2}\lesssim \|a(0)\|_{L_{x}^{2}}-\|Q_{0}\|_{L_{x}^{2}}.
 \]
  Also note the conditions $b_{s}\sim- \Gamma_{b}$ essentially ensures $b$ stays positive for all time.
\end{rem}

\section{Probabilistic Local Wellposedness} \label{sec:pi}
We note that while the cubic nonlinear Schr\"odinger equation \eqref{eq:nls} is \emph{deterministically} well-posed in $L^2_x(\RRR^2)$, we are seeking nonlinear smoothing and quantitative estimates, which are not true for general deterministic data. As such, we exploit several properties of the free evolution of the random data, as well as multilinear estimates involving such random functions.

The analysis in this section has many similarities to the random data analysis of Bourgain from \cite{B96}. Indeed, our choice of function to randomize is intended to mimic the random data appearing in \cite{B96}. However, several new ingredients are needed in our analysis to carry these estimates, and in particular, we need some new arguments in order to adapt Bourgain's result to the non-compact setting.

\medskip
For technical reasons, we fix $\varepsilon_0 > 0$ and $b=\frac{1}{2}+\eez$. We will also fix
\begin{equation}\label{eq: puretr}
\eez\ll\eeo\ll\eet\ll \delta\ll s\ll 1.
\end{equation}
and any $\epsilon$ involved in the analysis should satisfy $\epsilon\ll \epsilon_{0}$. 

\begin{rem}
One may assume, for example, $\eez\ll \eet^{s/10}$.  The purpose of these parameters are to overcome a technical issue arising from the scaling of $X^{s,b}$ spaces, specifically letting  $h_{\lambda}:=\frac{1}{\lambda}h(\frac{t}{\lambda^{2}}, \frac{x}{\lambda})$, and $b=\frac{1}{2}+\epsilon_{0}$, one has
\[
\|h_{\lambda}\|_{X^{s,b}}\lesssim \|h\|_{X^{s,b}}\lambda^{-s-2\epsilon_{0}}.
\]
We note that in general, the scaling properties of $X^{s,b}$ do not pose problems since our local wellposedness and energy estimates will be sub-critical in nature, and we do not need to derive end-point type  estimates where $\epsilon$ losses would be forbidden. 
\end{rem}

The aim of the current section is to establish improved\footnote{It is standard that the problem we treat in this article is deterministically locally well-posed with intervals of length $\sim \lambda(t_{k})^{-2+\eet}$, if one only care $L^{2}$ level well-posedness} probabilistic local wellposedness. Specifically, we will establish a result analogous to \cite[Lemma 3.3]{colliander2009rough} with randomized  data. 

In \cite{colliander2009rough}, every LWP interval $[t_{k}, t_{k+1}]$ is split into intervals $\cup_{j}[\tau_{k}^{j}, \tau_{k}^{j+1}]$ such that 
\[
|\tau_{k}^{j}-\tau_{k}^{j+1}|\sim \lambda(\tau_{k}^{j})^{-2}\sim \lambda(t_{k})^{-2}\sim 2^{-2k}.
\]
Due to the aforementioned technical issues relating to the scaling of $X^{s,b}$ spaces, we will instead split $[t_{k}, t_{k+1}]$ into intervals
$\cup_{j=1}^{J_{k}}[\tau_{k}^{j}, \tau_{k}^{j+1}]$, so that 
\begin{equation}\label{eq: sizeoflwp}
|\tau_{k}^{j}-\tau_{k}^{j+1}|\sim \lambda(t_{k})^{-2+\eet}
\end{equation}
and note there are at most $k\lambda(t_{k})^{-\eet} $ such many LWP intervals within $[t_{k}, t_{k+1}]$, thanks to the bootstrap assumption \eqref{eq: balwpinteval}.

Thus, let $0 < T \ll 1$ and let $I=[\tau_{k}^{j}, \tau_{k}^{j+1}]\subset [t_{k}, t_{k+1}]\subset [0, T]$ with $|I|\sim \lambda(t_{k})^{2-\eet}$. Recall $u$ solves 
\begin{equation}
	\begin{cases}
	iu_{t}+\Delta u=-|u|^{2}u, \quad (x,t)\in \mathbb{R}^{2}\times I,&\\
	u(\tau_{k}^{j})=a(\tau_k^{j})+F(\tau_{k}^{j}),&
	\end{cases}
\end{equation}
where $a$ is of form \eqref{eq: ans} and satisfies the bootstrap assumption \eqref{eq: bah1control}. We note that \eqref{eq: bah1control} implies
\begin{equation}
	\|a(\tau_{k}^{j})\|_{H^{s}}\sim \frac{1}{\lambda^{s}(\tau_{k}^{j})}\sim \frac{1}{\lambda^{s}(t_{k})}.
\end{equation}

We now turn to the statement of the probabilistic local estimates.

\begin{lem}\label{lem: lwplocalpro}
Let $f^\omega$ be the randomization defined in \eqref{eq: randomized}. Fix $p = \infty-$, and $q =4$, and let $\Sigma_1\subseteq \Omega$ be a subset so that \eqref{ell_infty}, \eqref{Lp} and \eqref{Lp2} hold. Then there exists a set $\Sigma_2 \subseteq \Omega$ satisfying
\[
\mathbb{P}(\Sigma_2^c) \lesssim e^{-|\tau_{k}^{j}-\tau_{k}^{j+1}|^{-c}}
\] 
so that for every $\omega \in \Sigma_1 \cap \Sigma_2$, if $u$ solves \eqref{eq:nls} with initial data $a_0 + f^\omega$, then one has 
\begin{equation}\label{eq: lwpae}
\|a\|_{X^{s,b}[I]}=\|u-F\|_{X^{s,b}[I]}\lesssim \frac{1}{\lambda(\tau_k)^{s+\epsilon_{1}}},
\end{equation}
and
\begin{equation}\label{eq: lwpinae}
\|I_{N(\tau_{k}^{j})}a\|=\|I_{N(\tau_{k}^{j})}(u-F)\|_{X^{1,b}[I]}\sim \frac{1}{\lambda(t_{k})^{1+\epsilon_{1}}}.
\end{equation}
\end{lem}
Note that in particular, \eqref{eq: lwpinae} implies 
\begin{equation}\label{eq: directmodify}
\|I_{N(T)}a\|_{X^{1,b}[I]}\lesssim \left(\frac{N(T)}{N(t_{k})}\right)^{1-s}\left(\frac{1}{\lambda(t_{k})}\right)^{1+\epsilon_{1}},
\end{equation}
see \cite[(3.14) and (3.20)]{colliander2009rough}.

\begin{rem}\label{rem:prob_bds_no_lambda}
We note that we may establish an identical result and additionally obtain that
\[
\mathbb{P}(\Sigma_1^c) \lesssim e^{-1/\lambda_0^c}.
\]
for some $c > 0$. Indeed, fix $p = \infty- $, and let $\Sigma_1 \subseteq \Omega$ be such that \eqref{ell_infty}, \eqref{Lp} and \eqref{Lp2} hold with constant $\sim \lambda_0^{-c_1}$ for some $c_1 > 0$ small. Then up to redefining $\varepsilon_1$, we are able to absorb this additional loss into estimates \eqref{eq: lwpae} and \eqref{eq: lwpinae}. We additionally note that such a subset is independent of $k$ and $I = [\tau_k^j, \tau_k^{j+1}]$, Hence, to simplify our arguments, we will instead assume \eqref{ell_infty}, \eqref{Lp} and \eqref{Lp2} hold with a fixed ($\lambda_0$-independent) constant.
\end{rem}

\begin{rem}\label{rem: prolemworks}
On the whole interval $[0,T]$, the set  that one needs to drop arising from the subset $\Sigma_2$ in Lemma \ref{lem: lwplocalpro} contributes total probability\footnote{The $c$ may change line by line and may not be the same as in Lemma \ref{lem: lwplocalpro}.} bounded by
\begin{equation}
\sum_{k=k_{0}}^{k_{+}}ke^{-2^{ck}}\leq \sum_{k=k_{0}}^{k_{+}}ke^{-2^{ck}}\lesssim Ce^{-2^{\frac{c}{2}}k_{0}}
\end{equation} 
Thus, by making $k_{0}$ large enough (i.e. making $\lambda_{0}$ small enough), one can ensure that up to a set of small probability, for every $I=[\tau_{k}^{j}, \tau_{k}^{j+1}]\subset [t_{k}, t_{k+1}]\subset [0,T]$, the conclusion of Lemma \ref{lem: lwplocalpro} holds.
\end{rem}

\begin{rem}\label{rem: xsb}
As we will see, the proof reduces to controlling the nonlinearity $|a + F|^{2}(a+F)$, and in particular, the term $|F|^{2}F$ is the most difficult to control. 

While the term $|a|^{2}a$ essentially follows from standard deterministic theory, we need to  introduce parameters $\epsilon_{0}, \epsilon_{1}, \epsilon_{2}$ in \eqref{eq: puretr} for the following reasons: we will need to rescale $a$ to $\lambda a(\lambda^{2}t, \lambda x)$ so that it is normalized in $X^{s,b}$, perform the standard the deterministic local theory, and then scale back. This generates an extra error $\lambda(t_{k})^{-C\epsilon_{0}}$ due to the fact $X^{s, b}$ is not scale invariant, resulting in an extra loss $\lambda(t_{k})^{-\epsilon_{1}}$.

The term $|a|^{2}F$ also essentially follows from deterministic local theory since we are able to distribute derivatives using bilinear Strichartz estimates, and have sufficient smoother functions to do so. However, we need the smallness of the interval to close this estimate, and hence we shrink the interval by an extra $\lambda(t_{k})^{\epsilon_{2}}$ factor. We will not focus too intently on these parameters since we wish to emphasize the treatment of the terms $|F|^{2}F$ and $|F|^{2}a$  (up to complex conjugates), however we point of that any loss of form of $\lambda(t_k)^{-\epsilon_{2}}$ is acceptable in the estimates because of the smallness of the time interval.
\end{rem}

\begin{proof}[Proof of Lemma \ref{lem: lwplocalpro}]
Recall that we use the ansatz \eqref{eq: ans}, and that $a$ solves the difference equation
\begin{equation}\label{eq: diffa}
\begin{cases}
ia_{t}+\Delta a=|a + F|^{2}(a+F),\quad t\in [\tau^{j}_{k},\tau_{k}^{j+1}],\\
\| a(\tau_k, x)\|_{H^s} \sim \lambda(\tau_k)^{-s} .
\end{cases}
\end{equation}
Without loss of generality and by a time translation, we may take $\tau_{k}^{j}=0$. Let $\eta(t)$ be a smooth cut-off, with  $\eta(t)\equiv1$ when $|t|\leq 1$, and $\eta(t)\equiv 0,$ for $|t|\geq 2$. Let $\eta_{\beta}(t)=\eta(t/\beta)$. We denote by $\widetilde{a}$ the extension of $a$ to the real line. By Duhamel's formulation, we need to estimate
\[
a(t,x) = e^{i(t - \tau_k)}a(\tau_k, x)  - i \int_{\tau_k}^t e^{-i(t - s)} \bigl(|a + F|^{2}(a+F) \bigr) ds,
\]

The linear part of $a(t,x)$ can be handled with the standard $X^{s,b}$ estimate, using the form of  $a(\tau_{k}^{j})$. For the inhomogeneous nonlinear estimate we need to control
\begin{equation}\label{eq: Xsbsmoothing}
\begin{aligned}
&\left\|\int_{\tau_k}^t e^{-i(t - s)} \bigl(|a + F|^{2}(a+F) \bigr) ds \right\|_{X^{s, b}(I)} \\
&\lesssim \left\| \eta_{|I|} \int_{\tau_k}^t e^{-i(t - s)} \bigl(|\widetilde{a} + F|^{2}(\widetilde{a}+F) \bigr) ds \right\|_{X^{s, b}} \\
&\lesssim \left\|  \eta_{I}(t) | \widetilde{a} + F|^{2}( \widetilde{a}+F) |\right\|_{X^{s, b-1}} \\
&\lesssim \left\| \eta_{I}(t) (| \widetilde{a} + F|^{2}( \widetilde{a}+F) - |F|^2 F) \right\|_{X^{s, b-1}}  + \left\| \eta_{I}(t) |F|^2 F \right\|_{X^{s, b-1}} 
\end{aligned}
\end{equation}

Our main goal is to prove that given $\tilde{a}$ which satisfies 
\[
\|\tilde{a}\|_{X^{s,b}}\lesssim \left(\frac{1}{\lambda(\tau_{k}^{j})}\right)^{s},
\]
one has that \eqref{eq: Xsbsmoothing} is bounded 
\begin{equation}\label{eq: goal1}
\frac{1}{2} \left(\frac{1}{\lambda(\tau_{k}^{j})}\right)^{s}.
\end{equation}
The extra $1/2$ factor yields, in the usual manner, that the solution map is a contraction. Indeed, this follows from the fact that the $X^{s,b}$  and random data analysis involved is sub-critical in nature, and that we are working on a small interval. We note that in order to establish \eqref{eq: goal1}, it will suffice to prove that \eqref{eq: Xsbsmoothing} is bounded by 
\begin{equation}\label{eq: goal2work}
\left(\frac{1}{\lambda(\tau_{k}^{j})}\right)^{s+\epsilon_{1}}
\end{equation}
since we are  working on intervals of (extra) small length, and the extra smallness of time interval $\lambda(t_{k})^{\eet}$ will be able to beat the $\epsilon_{1}$ loss, as remarked in Remark \ref{rem: xsb}. In the sequel, we will not distinguish between $a$ and $\tilde{a}$, since they will be treated and estimated in a same way.  We finally remark, there is a simple way to gain smallness of $X^{s,b}$ by localizing time, i.e. to trade part of $b$ derivative to estimate $X^{s,b'}$ for some $\frac{1}{2}<b'<b$. This will \emph{never} be involved in our analysis, however, because the maximum allowable difference between $b$ and $b'$ is bounded by $\epsilon_{0}$, which is too small to over come the extra loss in \eqref{eq: goal2work}.

\medskip
Thus, we focus on establishing \eqref{eq: goal2work}, and we begin with the term $\eta_{I}(t)| F|^2F$. In light of our discussion in Remark \ref{rem:prob_bds_no_lambda}, we will prove that 
\begin{equation}\label{eq: fgood}
\left\|  \eta_{|I|}(t)|F|^2 F \right\|_{X^{s, b-1}} \lesssim 1.
\end{equation}
Moreover, for the majority of the proof, we will in fact prove that 
\begin{equation}
\left\| \eta(t) |F|^2 F \right\|_{X^{s, b-1}} \lesssim 1,
\end{equation}
and we note that we can replace the term $\eta_{|I|}(t)$ with $\eta(t)$ since $X^{s,b}$ spaces are well behaved under time localization. Additionally, we will occasionally abuse notation, and use that $\eta^{3}\simeq\eta$, which will enable us to replace $F$ with $\eta(t) F$ as needed. Time localization is only needed when we argue that the extra subset we drop has probability $\lesssim e^{-|\tau - \tau_{j+1}|^{-c}}$. We will revisit this later.

\medskip
Let
\[
* = \{n_1, n_2, n_3\in \mathbb{Z}^{2} \,:\, n_2 \neq n_1, n_3\}
\]
and set
\[
h_k = e^{it \Delta} \widecheck{\psi}_k,
\]
then $|F|^2F$ can be written as 
\begin{align*}
&\underbrace{\sum_{*}  \iiint g_{n_1} \overline{g}_{n_2} g_{n_3} f_{n_1}\overline{f}_{n_2}f_{n_3} \psi_{n_1}(\xi_1) \overline{\psi}_{n_2}(\xi_2) \psi_{n_3}(\xi_3) e^{i (\xi_1 -\xi_2 + \xi_3) x - (|\xi_1|^2 - |\xi_2|^2 + |\xi_3|^2) t }}_{\textup{Term 1}}\\
& - \underbrace{\sum_{n} \iiint |g_{n}|^2 g_{n} |f_n|^2 f_n \psi_{n}(\xi_1) \overline{\psi}_{n}(\xi_2) \psi_{n}(\xi_3) e^{i (\xi_1 -\xi_2 + \xi_3) x - (|\xi_1|^2 - |\xi_2|^2 + |\xi_3|^2) t}}_{\textup{Term 2}} \\
& + \underbrace{2 \sum_{n_1, n_3} |f_{n_1}|^2 |g_{n_1}|^2 |h_{n_1}|^2 f_{n_3} g_{n_3} h_{n_3}}_{\textup{Term 3}}.
\end{align*}
We estimate these terms separately, beginning with the easiest.
\subsection*{Term 2}
As in \cite{B96} ,  we will directly estimate the $L_{t}^{\infty}H_{x}^{s}$ norm of this term. A simple triangle inequality gives
\begin{equation}
\begin{aligned}
& \left\| \langle \xi \rangle^s  \sum_{n}  |g_{n}|^2 g_{n} |f_n|^2 f_n \iiint \delta(\xi - \xi_1 + \xi_2 - \xi_3 \psi_{k}(\xi_1) \psi_{k}(\xi_2) \psi_{k}(\xi_3)  \right\|_{L^2_\xi} \\
& \lesssim    \left( \sum_{n}  \langle n \rangle^s \langle n \rangle^{3\varepsilon} |n|^{-3} \right)\lesssim 1
\end{aligned}
\end{equation}
where we have used the $\ell^\infty$ Gaussian bound \eqref{ell_infty}. This is summable in dimension $d=2$ provided $s < 1$.

\subsection*{Term 1} 
This is the term which typically appears in random data analysis, and can usually be used to illustrate what improvements one obtains for random data, see \cite{B96} for more details. Here, since we are not working with the NLS on a (rational) torus, one cannot directly reduce the problem into the same counting estimates as \cite{B96}. On the other hand, since we are on Euclidean space, we can now take advantage of the bilinear Strichartz estimates in Lemma~\ref{prop:bilin}, to help with analysis.

\medskip
Below we assume $|\xi_i| = |k_i| + O(1) \sim N_i$ for $i = 1,2,3$ and without loss of generality, set $N_{1}\geq N_{2}\geq N_{3}$, where $N_{i}$ are dyadic integers. We use notation $F_{i}:=F_{N_{i}}:=P_{N_{i}}F, i=1,2,3$. We first perform several reductions. We note that we may assume that
\begin{align}\label{eq:large_n3}
N_{3}\geq N_{1}^{99/100}
\end{align}
otherwise by bilinear Strichartz estimates, we obtain (recall $p$ is always large):
\begin{equation}
\begin{aligned}
\| \eta(t) F_1 F_2 F_3\|_{X^{s, b-1}}&\lesssim \sup_{\|h\|_{X^{s,1-b}}=1}\int F_{1}F_{2}F_{3}\bar{h} \\ & \lesssim N_1^s  \|F_1 F_3\|_{L_{t}^{2}L^2_x}  \|F_2\|_{L_{t,x}^{p}} \\
&\leq N_1^s   \left( \frac{N_3}{N_1} \right)^{1/2} \|F_3\|_{L^p_{t,x}} ,
\end{aligned}
\end{equation}
and using \eqref{Lp}, we may sum ($N_{1}\geq N_{2}\geq N_{3}$) provided
\[
N_1^{s - \frac{1}{2} + \frac{2}{p} } N_3^{\frac{1}{2} } \leq N_1^{s - \frac{1}{2}+ \frac{198}{100p} } N_1^{\frac{99}{200} } \lesssim 1, 
\]
which can be done by choosing $s > 0$ sufficiently small so that
\[
s - \frac{1}{200} + \frac{198}{100p} < 0.
\]

Proceeding, we will estimate this expression by reducing the problem into counting problems.  Following Bourgain, \cite{B96}, we start with a standard reduction. Here, we need to replace the $F$ by $\eta(t) F$. By definition of the $X^{s,b}$ space, we need to control
\begin{equation}\label{eq: fff}
\begin{aligned}
&\left\| \langle \tau - |\xi|^2 \rangle^{b-1} \langle \xi \rangle^s \mathcal{F}_{\tau,\xi}(\eta(t)\textup{Term 1})  \right\|_{L^2_\tau L^2_\xi},
\end{aligned}
\end{equation}
We let $\mu = \tau - |\xi|^2$ and we first claim we only need to control the region
\begin{equation}
\mu\ll N_{1}^{10s}.
\end{equation}
Indeed, one may use dual estimates to estimate \eqref{eq: fff}. For the deterministic theory, one needs to pair a function $h$ such that $\|h\|_{X^{0,1-b}}=1$, and the full $1-b=1/2+\epsilon_{0}$  ($X^{s,b}$ type) derivatives are needed, since one needs to control $\|h\|_{L_{t,x}^{4}}$.  Here, the random data allows us to beat the usual Strichartz estimates, and we are able to place each copy of $F$ into $L_{t,x}^{\infty-}$, and hence one only needs control of $\|h\|_{L_{t,x}^{3+}}$, which by interpolation only requires $1/3+\epsilon$ ($X^{s,b}$ type) derivatives. Thus, if one is in the case $\|\mu\|\geq N_{1}^{10s}$, the gain in the $X^{s,b}$ smoothing will compensate the $N_{1}^{s}$ loss in the space derivative. See equations \cite[(30) and (35)]{B96}. Moreover, we may focus on the case $\mu=O(1)$ and sum different part via triangle inequality, by paying extra $N_{1}^{Cs}$ loss, note $C$ will be large but we still have $Cs\ll 1$. 

Going back to \eqref{eq: fff}, we first expand 
$\mathcal{F}(\eta(t)\textup{Term 1})$
\begin{align}
& \iint e^{-i x \cdot \xi} e^{- i \tau t} \eta(t)\iiint  \psi_{n_1}(\xi_1) \psi_{n_2}(\xi_2) \psi_{n_3}(\xi_3) e^{i (\xi_1 -\xi_2 + \xi_3) x - (|\xi_1|^2 - |\xi_2|^2 + |\xi_3|^2) t}\\
& =  \iiint  \hat{\eta}(\tau -|\xi_1|^2 - |\xi_2|^2 + |\xi_3|^2) \delta (\xi - \xi_1 + \xi_2 - \xi_3) \psi_{n_1}(\xi_1) \psi_{n_2}(\xi_2) \psi_{n_3}(\xi_3).
\end{align}
We substitute this expression into \eqref{eq: fff}, and we recall that
\[
|\xi_{1}-\xi_{2}+\xi_{3}|^2 - |\xi_1|^2 + |\xi_2|^2 -  |\xi_3|^2 = 2 \langle \xi_2 - \xi_1, \xi_2 - \xi_3 \rangle.
\]

Now, proceeding we need to handle the estimates separately for $N_1 \geq N_{1,0}$, and $N_1 < N_{1,0}$ for some $N_{1,0}$ which we will determine below. When $N_1 \geq N_{1,0}$, this is where we drop the extra set of small probability, $\Sigma_2^c$, mentioned in the statement of the lemma. This extra argument is (more or less) standard, but we provide a sketch here. We fix such an $N_1$, and we use the multilinear Gaussian estimate of Lemma \ref{lem:multi_gauss} with constant $K = N_1^{Cs}$ to replace Term 1 by its $L^2_\omega$ norm by dropping an \emph{extra} set of probability $ \leq e^{-N_{1}^{c(\epsilon)}}$, where $c(\epsilon) > 0$ is a small $\epsilon$-dependent constant. Ultimately we need to control
\begin{equation}\label{eq: okcount}
N_{1}^{2Cs}\left( \sum_{*} \frac{1}{|n_1|^2} \frac{1}{|n_2|^2} \frac{1}{|n_3|^2}\right)^{1/2}
\end{equation}
where 
\[
* = \{n_1, n_2, n_3, n_2 \neq n_1, n_3, \langle n_{2}-n_{3}, n_{2}-n_{1}\rangle=O(1)\}.
\]
Recall that by restricting to the case $\mu=O(1)$ we lose an extra $N_{1}^{Cs}$, and we a priori have $|f_{n_{i}}| \leq \frac{1}{|n_{i}|}$. As in \eqref{eq:large_n3}, we only consider the case $N_{3}\geq N_{1}^{99/100}$. 

When $|N_{2}-N_{3}| < N_{3}^{1/10}$, for fixed $n_{1}, n_{2}$, there will be at most $N_{3}^{1/5}$ many $N_{3}$, and we may use that $N_3 \geq N_1^{\frac{99}{100}}$ to sum
\[
\sum_{N_1, N_2, N_3, \, |N_{2}-N_{3}| <  N_{3}^{1/10}}   N_{1}^{2Cs}\frac{1}{|n_1|^2} \frac{1}{|n_2|^2} \frac{1}{|n_3|^2} \lesssim  N_1^{2Cs}  N_1^{-(2 -\frac{1}{5}) \cdot \frac{99}{100} }.
\]
which is acceptable for $s$ sufficiently small.

When $|N_{2}-N_{3}|\geq N_{3}^{1/10}$, we mimic the counting in \cite[Lemma 1]{B96}.  Fixing $n_{2}$ and $n_{3}$, we note there could be at most $N_{1}^{2} / N_{3}^{1/10}$ many $n_1$. Indeed, let 
\[
n_{1}-n_{3}=(c_{1}, c_{2}), \quad n_{2}-n_{3}=(b_{1}, b_{2}),
\]
and assuming, for example, that $b_{2}\geq N_{3}^{1/10}$, and fixing $c_{1}$, there can be at most $N_{1}/N^{1/10}_{3}$ many $c_{2}$, and at most $N_{1}$ many $a_{1}$. Hence, we may bound \eqref{eq: okcount} by $N_{3}^{-1/10}$ or $N_{1}^{-1/20}$, since $N_3 \geq N_1^{99/100}$, and we obtain a bound which is summable $s$ sufficiently small. Since we must drop an extra subset for every fixed $N_1 \geq N_{1,0}$, after summation we have that the probability of the subset we drop is $\leq e^{-(N_{1,0})^c}$. 

\medskip
For $ N_{1}\leq N_{1,0}$, we use \eqref{ell_infty}, and then we argue in a purely deterministic manner, using the fact the interval is (extra) short, of length $\sim \lambda(t_k)^{-2+\epsilon_{2}} $ to close. Here, we need to use the cut-off $\eta_{|I|}(t)$. To close these estimates, we fix $N_{1,0}\sim \lambda(t_{k})^{-\tilde{c}(\epsilon)}$, where again $\tilde{c}(\epsilon) > 0$ is another small, $\epsilon$-dependent constant. This yields the stated bound on $\mathbb{P}(\Sigma_2^c)$, recalling how we defined the length of the time intervals. See also the discussion below \cite[(46)]{B96}.

\subsection*{Term 3}
This term is the most  distinct from the analysis in \cite{B96}.  Indeed, in \cite{B96}, a Wick-ordering is applied and this term does not appear at all.  We remark that one can still apply a phase transform to cancel this term, but such a phase, unlike \cite{B96}, will be a function rather than a number, and will not leave the NLS invariant. The key difference between our setting and Bourgain's is that our initial data lies at $L^2_x(\RRR^2)$ regularity, and hence we do not have to control the same divergences which appear for data in the support of the invariant Gibbs measure considered by Bourgain.

We recall that we are considering
\[
\sum_{n_1, n_3} |f_{n_1}|^2 |g_{n_1}|^2 |h_{n_1}|^2 f_{n_3} g_{n_3} h_{n_3}
\]
we let
\[
\widetilde{\theta}(t,x,\omega) = 2 \sum_{n_1} |f_{n_1}|^2 |g_{n_1}|^2 |h_{n_1}|^2 
\]
and note that this term is equal to
\[
\widetilde{\theta}(t,x,\omega)  F.
\]
Moreover, we observe that 
\[
\mathbb{E} \left( \sum_{n_1} |f_{n_1}|^2 |g_{n_1}|^2 \right) < \infty,
\]
and hence almost surely,
\[
\{f_{n_1} g_{n_1}  \}_{n_1 \in \mathbb{Z}^2} \in \ell^2
\]
and up to an exceptional set from \eqref{ell_infty}, we have that
\[
 |f_{n_1} g_{n_1}| \lesssim \frac{|n_1|^\varepsilon}{|n_1|}.
\]
Now, observing that
\[
|h_{k}|^2 = e^{it\Delta} \widecheck{\psi}_k \,\overline{e^{it\Delta} \widecheck{\psi}_k} ,
\]
and using that the free evolution does not affect the Fourier support, for each $k$, this term $|h_k|^2$ is supported in a ball of size two around the origin by convolution of the supports.  And one indeed have $|h_{k}|^{2}=|h_{0}(t,x-kt)|^{2}$ and $h_{0}$ is smooth. 
Thus, $\theta(x,t)$ is also smooth since $\sum |f_{k}|^{2}\lesssim 1$.

\medskip
We need to estimate 
\[
 \| \widetilde{\theta}(t,x,\omega)  F \|_{L^2_t H^s_x }.
\]
and in light of the observations above, it suffices to estimate the expression
\[
\| \widetilde{\theta}(t,x,\omega)  |\nabla|^s F \|_{L^2_t L^2_x }.
 \]
First observe since $|h_{k}|^{2}$ are all frequency localized around $1$, we apply $L^{2}$-orthogonality to 
derive
\begin{equation}
\begin{aligned}
&\| \widetilde{\theta}(t,x,\omega)  |\nabla|^s F \|_{L^2_t L^2_x }
\lesssim 
  &\sum_{k  \in \mathbb{Z}^2} \| \widetilde{\theta}(t,x ) |\nabla|^s P_k f \|_{L^2_t L^2_x}^2 
  \end{aligned}
\end{equation}
Now, noting that $h_\ell$ enjoys unit-scale Bernstein estimates (and hence lies in $L^\infty_{t,x}$) we obtain
\begin{align}
& \sum_{k  \in \mathbb{Z}^2} \| \widetilde{\theta}(t,x )e^{it\Delta} P_{k}f\|_{L^2_t L^2_x}^2 \\
& \lesssim \sum_{k}  \left( \sum_{\ell } |f_\ell g_\ell |^2 \|h_\ell\|_{L^\infty_{t,x}} \| e^{it \Delta} \widecheck{\psi_\ell}  \, e^{it\Delta} (g_k\psi_{k})(\omega) |k|^{s} f_{k}\|_{L^2_t L^2_x} \right)^2.
\end{align}
Now, we apply bilinear Strichartz estimate of Proposition \ref{prop:bilin}, and conjugate with Galilean symmetry for 
\[
e^{it \Delta} \widecheck{\psi_\ell}  \, e^{it\Delta} (g_k\psi_{k})(\omega),
\]
and plugging in the $\ell^\infty$ Gaussian bound \eqref{ell_infty}, we can estimate this expression by
\begin{align}
\sum_{k} |f_{k}|^2  \left( \sum_{\ell} |f_\ell |^2 |k|^{s + \varepsilon} \frac{1}{\langle k - \ell \rangle^{1/2}}  \right)^2.
\end{align}
Note that $|f_{k}|^{2}$ is summable in $k$.

Now, fix $k$, then if $|k- \ell| > k / 2$, then provided $s + \varepsilon < \frac{1}{2}$, this expression is bounded. Alternatively, if $|k - \ell| \leq k / 2$, then we use the fact that $\| |\ell | f_\ell \|_{\ell^\infty} \leq C$ and that $\ell \sim k$ to obtain
\begin{align}
 &\sum_{\ell, \, |\ell - k | \leq k /2} |f_\ell |^2 |k|^{s + \varepsilon} \frac{1}{\langle k - \ell \rangle^{1/2}} \\
 &\lesssim \sum_{\ell, \, |\ell - k | \leq k /2} k^{-2} |k|^{s + \varepsilon} \frac{1}{\langle k - \ell \rangle^{1/2}} \\
 & \simeq k^{-2} |k|^{s + \varepsilon} |k|^{3/2},
\end{align}
which is bounded provided, again, that $s + \varepsilon < 1/2$.

\medskip
We now attend to the terms involving $\tilde{a}$. For notational convenience, we will still use $a$ to denote $\tilde{a}$. Again, we will use the notation $a_{i}:=a_{N_{i}}:=P_{N_{i}}a$, and $N_{i}$ is a dyadic integer, and similarly for $F_{j}, j=1,2,3$. We will have to deal with multiple cases, depending on the frequency at which the random function is appearing.  Before proceeding, we note that since we are working on a interval with length $\lesssim \lambda(t_{k})^{2}$, and since by assumptions on the subset of the probability space, we can use H\"older's inequality in time to break the scaling and derive, for example,
\begin{equation}\label{eq: verylocal}
\|F\|_{L_{t,x}^{4}[I]}\lesssim |I|^{1/4-}\lesssim |\lambda(t_{k})|^{1/2-}
\end{equation}
This will be frequently used in the analysis below.

\subsection*{Case 1: $\|F_{N_1} F_{N_2} a_{N_3} \|$}

By duality, we estimate the expression
\[
\sum_{N_1\gg N_2 \geq N_3} N_{1}^{s}\int F_{1}F_{2}a_{3}h
\]
for $h \in X^{0, 1-b}$. When $N_1 \sim N_2$, we estimate this expression using bilinear Strichartz estimates and Cauchy-Schwarz in the highest frequency:
\begin{align}
&\sum_{N_1 \sim N_2 \geq N_3} \left(\frac{N_3}{N_1} \right)^{\frac{1}{2} - s} \|F_{1} \|_{L^\infty_t L^2_x} \|F_{2} \|_{L^{4+}_{t,x}}  N_3^s \|a_{3} \|_{L^\infty_t L^2_x} \|h\|_{L^{4-}_{t,x}} \\
&\lesssim \sum_{N_1 \sim N_2 \geq N_3} \left(\frac{N_3}{N_1} \right)^{\frac{1}{2} - s} \|F_{1} \|_{L^\infty_t L^2_x} \lambda(t_k)^{\frac{1}{2} - } \|F_{2}\|_{L^{\infty-}_tL^{4+}_{x}}  N_3^s \|a_{3} \|_{L^\infty_t L^2_x} \|h\|_{L^{4-}_{t,x}} \\
\end{align}
which is summable.

When $N_{1}\gg N_{2}$, we use duality with $h \in X^{0, 1-b}$, and we decompose $h$ into dyadic blocks $h_{N_{4}}$, and now we have $N_{4}\sim N_{1}$, and once again by bilinear Strichartz estimates
\begin{align}
&\sum_{N_1\gg N_2 \geq N_3} N_{1}^{s}\int F_{1}F_{2}a_{3}h \\
&\lesssim \sum_{N_4 \sim N_1 \gg N_2 \geq N_3} \left(\frac{N_{3}}{N_{1}}\right)^{\frac{1}{2}-s}\|F_{1}\|_{L^\infty_t L^2_{x}}\|h_{4}\|_{X^{0,1-b}}\|F_{2}\|_{L^{4+}_{t,x}}N_3^s \|a_{3}\|_{X^{s, b}}\\
&\lesssim \sum_{N_4 \sim N_1 \gg N_2 \geq N_3} \left(\frac{N_{3}}{N_{1}}\right)^{\frac{1}{2}-s}\|F_{1}\|_{L^\infty_t L^2_{x}}\|h_{4}\|_{X^{0,1-b}} \lambda(t_k)^{\frac{1}{2} - } \|F_{2}\|_{L^{\infty-}_tL^{4+}_{x}}N_3^s \|a_{3}\|_{X^{s, b}}
\end{align}
This is again summable.
\subsection*{Case 2: $\|a_{N_1} F_{N_2} F_{N_3} \|$}
Can be estimated precisely as in previous estimate, but we don't need to transfer regularity through bilinear Strichartz.
\subsection*{Case 3: $\|F_{N_1} a_{N_2} F_{N_3} \|$}

When $N_1 \sim N_2$, We estimate using duality and bilinear Strichartz:
\begin{align}
&\sum_{N_1 \sim N_2 \geq N_3} N_{1}^{s}\int F_{1}a_{2}F_{3}h\\
&\lesssim \sum_{N_1 \sim N_2 \geq N_3} N_{1}^{s} \|F_{1}h\|_{L_{t,x}^{2}}\|a_{2}F_{3}\|_{L_{t,x}^{2}}\\
&\lesssim \sum_{N_1 \sim N_2 \geq N_3} N_{1}^{s}N_{2}^{-s}\|F_{1}\|_{L^{4+}_{t,x} }\|a_{2}\|_{X^{s,b}}\|F_3\|_{L^\infty_t L^2_x} \left(\frac{N_{3}}{N_{2}}\right)^{1/2} \\
&\simeq \sum_{N_1 \sim N_2 \geq N_3} N_{1}^{s}N_{2}^{-s}\|F_{1}\|_{L^{4+}_{t,x} }\|a_{2}\|_{X^{s,b}}\|F_3\|_{L^\infty_t L^2_x} \left(\frac{N_{3}}{N_{1}}\right)^{1/2} \\
&\lesssim \sum_{N_1 \sim N_2 \geq N_3} N_{1}^{s}N_{2}^{-s} \lambda(t_k)^{\frac{1}{2} -} \|F_{1}\|_{L^{\infty-}_t L^{4+}_{x}} \|a_{2}\|_{X^{s,b}}\|F_3\|_{L^\infty_t L^2_x} \left(\frac{N_{3}}{N_{1}}\right)^{1/2} \\
\end{align}
and we can sum this expression.

When $N_1 \gg N_2$,  we use duality with $h \in X^{0, 1-b}$, and we decompose $h$ into dyadic blocks $h_{N_{4}}$, and now we have $N_{4}\sim N_{1}$. We estimate:
\begin{align}
&\sum_{N_4 \sim N_1\gg N_2 \geq N_3} N_{1}^{s}\int F_{1}a_{2}F_{3}h_4.
\end{align}
We pair $a_2$ with either $F_1$ or $F_3$ depending on the value of
\[
\min \left( \left( \frac{N_2}{N_1} \right), \left( \frac{N_3}{N_2} \right)\right),
\]
using the other $F$ factor to estimate with $h$ as above.

For example, supposing we perform the bilinear Strichartz with $a_2F_3$, (in the case $N_{3}/N_{2}\leq N_{2}/N_{1}$) we then obtain
\begin{align}
&\sum_{N_4 \sim N_1 \gg N_2 \geq N_3 } N_1^s N_2^{-s} \left( \frac{N_3}{N_2} \right)^{1/2}  \|F_{3}\|_{L^\infty_t L^2_{x}} \|a_{2}\|_{X^{s, b}} \|h_{4}\|_{X^{0,1-b}}\|F_{1}\|_{L^{4+}_{t,x}}\\
&=  \sum_{N_4 \sim N_1 \gg N_2 \geq N_3} N_1^s N_2^{-s} \left( \frac{N_3}{N_1} \right)^{1/4}  \|F_{3}\|_{L^\infty_t L^2_{x}} \|a_{2}\|_{X^{s, b}} \|h_{4}\|_{X^{0,1-b}}\|F_{1}\|_{L^{4+}_{t,x}}\\
&\lesssim  \sum_{N_4 \sim N_1 \gg N_2 \geq N_3} N_1^s N_2^{-s} \left( \frac{N_3}{N_1} \right)^{1/4}  \|F_{3}\|_{L^\infty_t L^2_{x}} \|a_{2}\|_{X^{s, b}} \|h_{4}\|_{X^{0,1-b}}\lambda(t_k)^{\frac{1}{2} -} \|F_{1}\|_{L^{\infty-}_t L^{4+}_{x}}
\end{align}
where we have used that
\[
\min(a,b) \leq \sqrt{ab}.
\]
Once again this is summable for $s < 1/4$ using Cauchy-Schwarz in $N_1 \sim N_4$.

\subsection*{Case 4: $\|F_{N_1} a_{N_2} a_{N_3} \|$}
Once again, we estimate by duality. If $N_1 \sim N_2$, we have
\[
\sum_{N_1\gg N_2 \geq N_3} N_{1}^{s}\int F_{1}a_{2}a_{3}h
\]
and we estimate using bilinear Strichartz with $a_2 a_3$:
\begin{align}
&\sum_{N_1\gg N_2 \geq N_3} N_{1}^{s} N_2^{-s} N_3^{-s} \left( \frac{N_3}{N_2}\right)^{1/2} \|F_1\|_{L^{4+}_{t,x}} \|a_{2}\|_{X^{s, b}} \|a_{3}\|_{X^{s, b}}  \|h\|_{X^{0,1-b}}\\
& \lesssim \sum_{N_1\gg N_2 \geq N_3} N_{1}^{s} N_2^{-s} N_3^{-s} \left( \frac{N_3}{N_2}\right)^{1/2} \lambda(t_k)^{\frac{1}{2} -} \|F_{1}\|_{L^{\infty-}_t L^{4+}_{x}} \|a_{2}\|_{X^{s, b}} \|a_{3}\|_{X^{s, b}}  \|h\|_{X^{0,1-b}}
\end{align}
which is summable using Cauchy-Schwartz in $N_2 \sim N_1$.

When $N_2 \ll N_1$, we dyadically decompose $h$ into $h_{N_4}$ and note we must have $N_1 \sim N_4$. We use bilinear Strichartz between $F_1$ and $a_3$, and we put $a_2 \in L^{4+}_{t,x}$, and we obtain
\[
\sum_{N_4 \sim N_1 \gg N_2 \geq N_3 } N_1^s  \left( \frac{N_3}{N_1} \right)^{1/2}  N_3^{-s} \|a_{3}\|_{X^{s, b}} \|a_{2}\|_{X^{\epsilon, b}} \|h_{4}\|_{X^{0,1-b}}\|F_{1}\|_{L^\infty_t L^2_x},
\]
which is again summable.  Note that we do lose an extra $\lambda^{-\epsilon}$ in the term  $\|a_{2}\|_{X^{\epsilon, b}}$.

\subsection*{Case 5: $\|a_{N_1} F_{N_2} a_{N_3} \|$}
We estimate as in the previous case, but do not need to transfer regularity from the function at the lowest frequency to the highest.
\subsection*{Case 6: $\|a_{N_1} a_{N_2} F_{N_3} \|$}
We estimate as in the previous case, but do not need to transfer regularity from the function at the lowest frequency to the highest.

\subsection*{Case 7: $\|a_{N_1} a_{N_2} a_{N_3} \|$}
As in standard deterministic local theory.

\subsection{Estimate of \eqref{eq: lwpinae}}
The estimate \eqref{eq: lwpinae} essentially follows directly from \eqref{eq: lwpae}, but we sketch the argument here. We will again write down the Duhamel formula of \eqref{eq: diffa}, and apply the $I$-operator $I_{N(\tau_{k}^{j})}$ on both sides, and estimate
\begin{equation}
\|I_{N}|F+a|^{2}(F+a)\|_{X^{s.b-1}}.
\end{equation} 
We distinguish four different scenarios:
\begin{enumerate}
\item[(i)] Three random pieces $|F|^{2}F$
\item[(ii)] Two random pieces terms, for example $F\overline{F}a$
\item[(iii)] Terms with at least two copies of $a$, and the highest frequency is on $a$, for example, the term $a_{1}\overline{F}_{2}a_{3}$
\item[(iv)] Terms where $F$ is at the highest frequency $F_{1}\overline{a}_{2}a_{3}$
\end{enumerate}
For situations (1) and (2), observe that $I_{N}$ will send $X^{s,b}$ into $X^{1,b}$ by losing 
\[
N^{1-s}(t_{k})\sim \lambda(t_{k})^{1-s}\lambda(t_{k})^{(1-s)\delta}.
\]
Using this estimate directly will miss the desired result by $\lambda(t_{k})^{(1-s)\delta}$, and we now detail how to recover this loss. 

In case (i), we see from the previous arguments for Terms 1, 2 and 3 that one beats the desired estimates by $\lambda(t_{k})^{-s}$, hence choosing $0 < \delta \ll 1$ small suffices. 

In case (ii), estimates of form \eqref{eq: verylocal} are applied and one gains a positive power of $\lambda(t)$, for example $\lambda(t_{k})^{1/100}$. Such gains are already enough to compensate  $\lambda(t_{k})^{(1-s)\delta}$ loss since $0 < \delta\ll 1$.

In case (iii), the estimate follows from standard deterministic arguments, and since the highest frequency is on $a$, thus $I_{N}(a_{1}\overline{b_{2}}b_{3})$ , (where $b=a$ or $F$) can be estimated (effectively) as $(I_{N}a_{1})\overline{b_{2}} b_{3}$, and standard persistence of regularity argument can close the estimates.

In case (iv), we are only concerned with the situation $N_{2}\ll N_{1}$, and further, one only needs to consider $N_{1}\geq N=N(\tau_{k}^{j})$. One can distinguish two subcases:
\begin{itemize}
\item  $N_{2}\geq N_{1}\lambda(t_{k})^{\tilde{\epsilon}}$, (it will be clear soon how should we choose this $\tilde{\epsilon}$)
\item $N_{2}\leq N_{1}\lambda(t_{k})^{\tilde{\epsilon}}$
\end{itemize}
In the first subcase, we again use persistence of regularity and transfer $\langle D\rangle I_{N}$ to $a_{2}$ by losing $(N_{1}/N_{2})^{s}$, and one will be able to close (recalling that an error of $\lambda(t_{k})^{-\epsilon_{1}}$ is allowed) if
\begin{equation}\label{eq: bbbb}
\tilde{\epsilon}s\lesssim  \frac{1}{10}\epsilon_{1}
\end{equation}
Note that the existence $\tilde{\epsilon}$ satisfying  \eqref{eq: aaaa} and \eqref{eq: bbbb} requires $\epsilon _{1}\geq \delta s^{2}$, which is acceptable.
In the second subcase,  one follows the same computations as with the term $F_{N_1} a_{N_2} a_{N_3}$ in Case 4, and the bilinear Strichartz estimates gives us an extra $(N_{2}/N_{1})^{1/4}\lesssim \lambda(t_{k})^{-\tilde{\epsilon}}$. We will again use the fact $I_{N}$ will send $X^{s,b}$ into $X^{1,b}$ by losing 
\[
N^{1-s}(t_{k})\sim \lambda(t_{k})^{1-s}\lambda(t_{k})^{(1-s)\delta},
\]
and we are able to close the estimates provided
\begin{equation}\label{eq: aaaa}
\tilde{\epsilon}\geq 10\delta s.
\end{equation}
This concludes the proof.
\end{proof}

\section{Energy Estimates}\label{sec: sig}
In this section, we combine the improved probabilistic local wellposedness in Section \ref{sec:pi} with the log-log bootstrap scheme, in particular \eqref{eq: balwpinteval}, to prove the analogue of \cite[Proposition 3.1]{colliander2009rough}. We still follow an I-method scheme, but our implementation has two main differences compared to \cite{colliander2009rough}.
\begin{itemize}
\item Our LWP theory is different from the standard $H^{s}(\RRR^2)$ lwp in \cite{colliander2009rough}.
\item The function $a$ will play the role of full solution $u$ in \cite{colliander2009rough}, and in particular $a$ does not solve the standard NLS, but rather a forced equation with random forcing terms, for which we need to incorporate extra random data type techniques into the I-method computation.
\end{itemize}
We note that we also take this opportunity to simplify certain aspects of the I-method arguments from \cite{colliander2009rough} in the current setting. Due to the fact that we ultimately combine the energy estimates with the log-log bootstrap, it seems unnecessary to exploit the full cancellation of the I-operator.

\medskip
Recall that we use the ansatz \eqref{eq: ans}. Let $J_{N(t)}$ denote the Fourier multiplier such that 
\begin{equation}
	J_{N(t)}+I_{N(t)}=Id.
\end{equation}
Following \cite{colliander2009rough}, let 
\begin{equation}
	\Xi(t)=\frac{\lambda^{2}}{2}\int |\nabla J_{N(t)}a(0)|^{2}dx,
\end{equation}

In the rest of the article, we will take $p = \infty -$, and we always assume as small probability set has already be dropped so that \eqref{ell_infty} , \eqref{Lp} and \eqref{Lp2} hold, and for every LWP inteval $[\tau_{k}^{j}, \tau_{k}^{j+1}]$, Lemma \ref{lem: lwplocalpro} holds. Since we discussed these considerations thoroughly in the previous section, we do not revisit them again. We will establish the following result.
\begin{prop}\label{pro: almost}
	Restricting the the subset so that \eqref{ell_infty}, \eqref{Lp}, \eqref{Lp2} and  Lemma \ref{lem: lwplocalpro} hold, we have the following: there exists some $\alpha_{1}>0$, such that for all $t\in [0,T]$, one has 
	\begin{equation}\label{eq: em}
		|E(I_{N(t)}a)+\frac{1}{\lambda^{2}(t)}\Xi(t)|\lesssim \biggl(\frac{1}{\lambda(t)}\biggr)^{2-\alpha_{1}},
	\end{equation}
and 
\begin{equation}\label{eq: pm}
	|P(I_{N(t)}a(t))|\lesssim \biggl(\frac{1}{\lambda(t)}\biggr)^{1-\alpha_{1}}.
\end{equation}
\end{prop}

\begin{rem}
	The exact value of $\alpha_{1}$ is somewhat different in our setting compared to \cite{colliander2009rough}. Indeed, recall that in \cite{colliander2009rough}, they establish a result for every $s>0$, and thus they have a choice of $\alpha_1$ for each such $s$. In contrast, we \emph{choose} some $0<s\ll 1$, and only need to find one such $\alpha_{1}$ for this particular $s$.
\end{rem}

\begin{proof}[Proof of Proposition \ref{pro: almost}]

We will focus on estimate \eqref{eq: em} and, as we will remark, \eqref{eq: pm} follows in a similar (if not simpler) manner. We will only estimate the case $t=T$ in Proposition~\ref{pro: almost}, and we will denote $N=N(T)$. The proof of \eqref{eq: em} has two parts:
\begin{itemize}
\item an initial estimate
\begin{equation}\label{eq: initialok}
|\Xi(T)+E(I_{N}(a(0)))|\lesssim \biggl(\frac{1}{\lambda(T)}\biggr)^{2-\alpha_{1}}, 
\end{equation}
for some $\alpha_{1}>0$, and
\item a growth estimate:
\begin{equation}\label{eq: growthok}
|E(I_{N}a(T))-E(I_{N}a(0))||\lesssim \biggl(\frac{1}{\lambda(T)}\biggr)^{2-\alpha_{1}}, 
\end{equation}
for some $\alpha_{1}>0$.
\end{itemize}

The initial estimate \eqref{eq: initialok}, this follows from the bootstrap assumptions \eqref{eq: bah1control} and the fact that the potential energy is subcritical compared to the kinetic energy, and one can argue exactly as the proof of (3.24) in \cite{colliander2009rough}. Thus, the rest of this section is mainly devoted to the proof of \eqref{eq: growthok}. Recalling \eqref{eq: eqfora}, we compute 
\begin{equation}
\begin{aligned}
\partial_t E(I_N a) 
& = \Re \int \overline{\bigl(I_N \Delta a +|I_{N}a|^{2}I_{N}a \bigr)} \left[ I_N (| a + F |^2 (a + F))  - |I_N a|^2 I_N a \right] \\
& := A_I + A_{II} + B_I + B_{II}
\end{aligned}
\end{equation}
where
\begin{align*}
A_I &:= \Re \int \overline{I_N \Delta a} \left[ I_N (| a + F |^2 (a + F)) - |I_N( a + F) |^2 I_N (a + F)  \right] \\
A_{II} &:= \Re \int \overline{I_N \Delta a} \left[  |I_N( a + F) |^2 I_N (a + F))  - |I_N a|^2 I_N a \right] \\
B_{I} &:= \Re \int \overline{|I_N a|^{2}I_{N}a}\left[ I_N (| a + F |^2 (a + F)) -  |I_N( a + F) |^2 I_N (a + F)  \right]  \\
B_{II} &:= \Re \int \overline{|I_N a|^{2}I_{N}a} \left[  |I_N( a + F) |^2 I_N (a + F)) - |I_N a|^2 I_N a \right].
\end{align*}

We will estimate each term separately, up to an observation on cancellation between terms that will be useful in the sequel. Indeed, we note the term $-|I_{N}F|^{2}I_{N}F$ in $A_{I}$ will cancel the same term in $A_{II}$, and the same cancellation also holds between $B_{I}$ and $B_{II}$. It is not immediately clear whether such cancellation is crucial, however it simplifies the analysis considerable because subtle probabilistic arguments, (as illustrated in previous section) have to be applied to analyze $|F|^{2}F$, resulting in extra subsets of small probability needing to be dropped. In order to redo the same estimates for all $N$, the analysis is not only more technical, but one needs to be careful about summability of the probabilities of these subsets, and the aforementioned cancellation frees us from this issue.

\medskip
We will now address the estimates of the terms $A_I, A_{II}, B_{I}, B_{II}$.

\subsection*{Estimates of $A_I$}
To estimate
\begin{align}\label{eq:AI}
A_I := \Re \int \overline{I_N \Delta a}\left[ I_N (| a + F |^2 (a + F)) - |I_N( a + F) |^2 I_N (a + F)    \right] ,
\end{align}
we will estimate the integral of $A_{I}$ within each LWP interval $[\tau_{k}^{j}, \tau_{k}^{j+1}]$ for $k_{0}\leq k\leq k_{1}, 1\leq j \leq J_{k}$ and then sum the resulting estimates. Due to the fact that our method is of subcritical nature, we need to beat the trivial estimate by at least $\lambda(t_{k})^{-\delta}$ within the interval $[\tau_{k}^{j}, \tau_{k}^{j+1}]$ With this in mind, we recall the parameters \eqref{eq: puretr}, and the fact that any loss of $\lambda(t_{k})^{-C\epsilon_{2}}$ or $\lambda(t_{k})^{-C\delta s}$ will be acceptable and can be neglected, we don't repeat this point later in the analysis.

We will see from the proof that if one fixes $k$, the estimate can be performed identically for different $j$. This follows since there are at most $k\lambda(t_{k})^{-\epsilon_{2}}\lesssim \lambda(t_{k})^{-2\epsilon_{2}}$ many LWP intervals. Hence, we can estimate a single LWP interval within $[t_{k}, t_{k+1}]$ and absorb the loss stemming from counting the number of intervals.  

Finally, one will observe that the estimate of $|A_{I}|$ is monotone in $k$, it is indeed enough to compute its integral in the last LWP interval $ [\tau_{k_{1}}^{J_{k-1}}, \tau_{k_{1}}^{J_{k}}]$ since $k_{1}\sim \ln \frac{1}{\lambda(t_{k})}$, any a loss of $k_{1}$ is also allowed by the previous analysis. At the heuristic level, one may compare it to summing up a geometric series, where the value of the sum is determined by the last term (up to an allowable error).

\medskip
We apply a Littewood-Paley decomposition to the quadrilinear term \eqref{eq:AI}, with frequencies $\xi_i \sim N_i$. We assume that $N_2 \geq  N_3 \geq N_4 $, and we write 
\[
a_i = a_{N_i} := P_{N_i} a,
\]
and similarly for $F_i$.   We will estimate two types of terms explicitly which are the most difficult cases. Other terms can be estimated via essentially the same (if not easier) analysis. As mentioned above, we will also exploit cancellation which enables us to handle some of the terms with three random pieces. We finally point out that when there are no random terms, one can just follow \cite{colliander2009rough}.

\subsection*{One random piece} The most difficult case is when the random term is at the highest allowable frequency, $N_2$. We consider one local wellposedness interval $I = [\tau_k^j, \tau_k^{j+1} ]$ and recall $|I| \lesssim  \lambda^2(t_k)$, and estimate
\begin{equation}\label{eq: target}
\sum_{N_1, N_2, N_3, N_4} \int_{\tau_{k}^{j}}^{\tau_{k}^{j+1}}\int \overline{I_N \Delta a_1 }\left[ I_N (F_2 \overline{a}_3 a_4) - I_N F_2 \overline{I_N a_3} I_N a_4    \right] dx dt.
\end{equation}
In order for the integral to be non-zero, we necessarily have $N_2 \gtrsim N$. We handle two cases:
\begin{itemize}
\item $N_1 \sim N_2$
\item $N_1 \ll N_2$.
\end{itemize}
Without loss of generality, assume that $N_i \geq 1$ for all $i$. When $N_{1}\sim N_{2}$, we will estimate the integral by estimating the term\footnote{Strictly speaking, \eqref{eq: target} is not bounded by \eqref{eq: like}, but by \eqref{eq: work}. What we mean here is, one can think about the estimate of \eqref{eq: target} as the estimate of \eqref{eq: like}, thus naturally lead to the estimate of \eqref{eq: work}.}
\begin{equation}\label{eq: like}
\left|\sum_{N_{1} \sim N_2 \geq N_{3} \geq N_{4} }  \iint N_{1}^{2}\bigl(\frac{N^{1-s}}{N_{1}^{1-s}} a_{1}\bigr)\bigl(\frac{N^{1-s}}{N_{2}^{1-s}}F_{2}\bigr) \overline{a_{3}}a_{4} \right |.
\end{equation}
We note that the complex conjugate will not be material and the quotients come from the definition of the $I_N$ operator. We use H\"older's inequality to estimate
\begin{equation}\label{eq: work}
N^{2-2s}\sum_{N_2 \sim N_{1}\geq N_{3}\geq N_{4}}N_{1}^{2s}\|a_1a_{3}\|_{L_{x,t}^{2}}\|F_{2}a_{4}\|_{L_{t,x}^{2}}.
\end{equation}

First we sum over $a_{4}$ via triangle inequality, noting that $\|a_{4}\|_{L_{t}^{\infty}L_{x}^{2^+}}\lesssim N_{4}^{-s+} \lambda(t_k)^{-2s}$ . As previously mentioned, we neglect any loss of form $\lambda(t_k)^{-\epsilon_{1}}$. Noting that in the current case $N_{1}\sim N_{2}$, and recalling that $F_{1} \in L_{t,x}^{p}=L_{t,x}^{\infty-}$ by \eqref{Lp2},  we have 
\begin{equation}
\begin{aligned}
&\lesssim N^{2-2s}\sum_{N_{1}\geq N_{3}\geq N_{4}} N_{1}^{2s}\|a_1a_{3}\|_{L_{t,x}^{2}} \|F_{1}\|_{L_{t,x}^{\infty-}}|I|^{1/2-}\\
&\sim N^{2-2s}\lambda(t_{k})^{1- 2s-}\sum_{N_{1}\geq N_{3}} N_{1}^{2s}\|a_1a_{3}\|_{L_{t,x}^{2}}\|F_{1}\|_{L_{t,x}^{p}}.
\end{aligned}
\end{equation}
Finally, we use the Bilinear Stirchartz esimates, to obtain
\begin{equation}\label{eq: a11inter}
\lesssim N^{2-2s}\lambda(t_{k})^{1-}\sum_{N_{1} \geq N_{3}}N_{1}^{s}\|a_1\|_{X^{0,b}}\|F_{1}\|_{L_{t,x}^{p}} N_{3}^{s}\|a_{3}\|_{X^{0,b}} \left( \frac{N_{3}}{N_{1}}\right)^{1/2-s}.
\end{equation}
Since both $N_{1}^{s}\|a_{1}\|_{X_{0,b}}$ (by definition) and $\|F_{1}\|_{L_{t,x}^{p}}$ (by \eqref{Lp2}) are $\ell_{2}$ summable, we may apply Cauchy-Schwarz inequality between these terms, and triangle inequality to sum $N_{3}\leq N_{1}$. We ultimately can estimate the contribution of this term to \eqref{eq: a11inter} by
\[
N^{2-2s}\lambda(t_{k})^{1- 2s-}.
\] 
Summing over all LWP intervals and applying \eqref{eq: balwpinteval}, \eqref{eq: sizeoflwp}, (recall also \eqref{eq: puretr}), one has
\begin{equation}
\begin{aligned}
|\int_0^{T}A_{I}|\lesssim \sum_{k_{0}}^{k_{1}}kN^{2-2s}\lambda(t_{k})^{1-2s-}\lambda(t)^{-\epsilon_{2}}\lesssim N^{2-2s}\lesssim \lambda(T)^{2(1+\delta)(1-s)}\lesssim \lambda(T)^{2-2s}
\end{aligned}
\end{equation}
which is the desired estimate.

\medskip
We also record the following simple observation from the computation above as a remark to reference later in the proof. We will not repeat the same argument later.
\begin{rem}\label{rem: wdww}
Provided we can estimate 
\[
\int_{\tau_{k}^{j}}^{\tau_{k}^{j+1}}A_{I} \lesssim N^{1-c_{0}s},
\]
for some $c_0 > 0$, for example $c_{0}=\frac{1}{100}$, we are able to sum the estimates up along all the LWP intervals.
\end{rem}

Next we turn to the case when $N_{1}\ll N_{2}$, we first observe the necessarily $N_{3}\sim N_{2}$. For notational convenience, we will use $S$ to denote $X^{0,b}$ in the rest of the section. We discuss two subcases:
\begin{itemize}
\item $N_{1}\geq N$,
\item $N_{1}\leq  N$.
\end{itemize}
First if $N_{1}\geq N$, we may reduce to estimating
\begin{equation}\label{eq: n1gn}
\begin{aligned}
&\left| \sum_{N_3 \sim N_{2}\geq N_{1}, N_{4}}\iint N_{1}^{2} \bigl(\frac{N^{1-s}}{N_{1}^{1-s}} a_1 \bigr) \bigl( \frac{N^{1-s}}{N_{1}^{1-s}} F_{2}\bigr) a_{3}a_{4} \right| \\
&\lesssim N^{2-2s} \left| \sum_{N_3 \sim  N_{2}\geq N_{1},N_{4}}\iint N_{1}^{2s}a_1F_{2}a_{3}a_{4} \right| \\
&\lesssim N^{2-2s}\sum_{N_3 \sim  N_{2}\geq N_{1},N_{4}}\lambda(t_{k})^{1-\varepsilon} N_{1}^{2s} \|a_1\|_{S} \|a_{3}\|_{S}N_2^\varepsilon \|F_{2}\|_{L_{t,x}^{\infty-} }N_{4}^{-s}\|a_{4}\|_{S}\\
&\lesssim N^{2-2s}\lambda(t_{k})^{1-\varepsilon} \sum_{N_3 \sim N_{2}\geq N_{1}}N_{1}^{2s} \|a_1\|_{S}N_2^\varepsilon \|a_{3}\|_{S}.
\end{aligned}
\end{equation}
Note that in the first line of  \eqref{eq: n1gn}, we either estimate $I_{N}(F_{2}a_{3}a_{4})$ whose out put frequency lies on $|\xi|\sim N_{1}$, or we estiamte $I_{N}F_{2}I_{N}a_{3}I_{N}a_{4}$, which has an I-operator smoothing at frequency at $N_{2}\geq N_{1}$. Ultimately we again obtain
\[
N^{2-2s}\lambda(t_{k})^{1-\varepsilon - 2s}
\]
which is enough from Remark \ref{rem: wdww}.

Finally, when $N_{1}\leq  N$, then we may reduce to estimate
 \begin{equation}
\left| \sum_{N_3 \sim  N_{2}\geq N_{1},N_{4},\, N \geq N_1} \iint N_{1}^{2} a_1 F_{2}N^{1-s}N_{2}^{-1+s}a_{3}a_{4} \right|,
 \end{equation}
 and in this case, one ends up with $\lambda(t_{k})^{1-2s - }N^{2-s}$, which is sufficient by Remark \ref{rem: wdww}.

 \subsection*{Two random pieces}
 We estimate the term with $a_{1}, F_{2}, F_{3}, a_{4}$. Recall, again, that unless $N_{2}\gtrsim N$ the expression we are estimating is zero. Once again, we handle two cases:
\begin{itemize}
\item $N_1 \sim N_2$
\item $N_1 \ll N_2$.
\end{itemize}

When $N_1 \sim N_2$, we estimate the integral via
 \begin{equation}
 \begin{aligned}
 &\sum_{N_2 \sim N_{1}\geq N_{3}\geq N_{4}}\int N_{1}^{2}\bigl(\frac{N^{1-s}}{N_{1}^{1-s}}a_1\bigr) \bigl(\frac{N^{1-s}}{N_{2}^{1-s}}F_{2}\bigr)F_{3}a_{4}\\
=&N^{2-2s}\sum_{N_2 \sim N_{1}\geq N_{3}\geq N_{4}}\int N_{1}^{2s}a_1F_{2}F_{3}a_{4}\\
\lesssim& N^{2-2s}\sum_{N_2 \sim N_{1}\geq N_{3}\geq N_{4}}N_{1}^{2s}\|a_1a_{4}\|_{L_{t,x}^{2}}\|F_{2}F_{3}\|_{L_{x,t}^{2}}.
\end{aligned}
 \end{equation}
 We apply bilinear Strichartz and use the $L_{t.x}^{\infty-}$ control from $F_{N_{1}}$ (using randomness), one has the above controlled by
 \begin{equation}
 \begin{aligned}
& \lesssim N^{2-2s}\sum_{N_{1}\sim N_2\geq N_{3}\geq N_{4}}N_{1}^{s}\|a_1\|_{S}\|F_{2}\|_{L^{\infty-}_t L^4_x}\|F_{3}\|_{L^{\infty-}_t L^4_x}|I|^{1/2-}N_{4}^{s}(\frac{N_{4}}{N_{1}})^{1/2-s}\|a_{4}\|_{S}\\
&\lesssim N^{2-2s}\lambda(t_{k})^{1-}\sum_{N_{1}\sim N_2\geq N_{3}\geq N_{4}}N_{1}^{s}\|a_1\|_{S}\|F_{2}\|_{L^{\infty-}_t L^4_x}\|F_{3}\|_{L^{\infty-}_t L^4_x}(\frac{N_{4}}{N_{3}})^{1/2-s}N_{4}^{s}\|a_{4}\|_{S}
 \end{aligned}
 \end{equation}
 Applying Cauchy Schwarz in $N_{4}$, one derives
 \begin{equation}
 \begin{aligned}
& \lesssim N^{2-2s}\lambda(t_{k})^{1-}\sum_{N_{1}\sim N_2 \geq N_{3}}N_{1}^{s}\|a_1\|_{S}\|F_{2}\|_{L^{\infty-}_t L^4_x}\|F_{3}\|_{L^{\infty-}_t L^4_x} \|a_3 \|_{S}(\frac{N_{3}}{N_{1}})^{1/2-s}\\
 &\lesssim N^{2-2s}\lambda(t_{k})^{1-} \sum_{N_{1} \sim N_2 \geq N_{3}}N_{1}^{s}\|a_1\|_{S}\|F_{2}\|_{L^{\infty-}_t L^4_x}\|F_{3}\|_{L^{\infty-}_t L^4_x} (\frac{N_{3}}{N_{1}})^{1/2-s}.
 \end{aligned}
 \end{equation}
To conclude, we sum over $N_{3}$, and then apply Cauchy Schwarz in $N_{1} \sim N_2$, which yields
 \[
 N^{2-2s}\lambda(t_{k})^{1-}.
 \]

Next, to estimate the expression when $N_{1}\ll N_{2}$, we again note that one necessarily has $N_{2}\sim N_{3}$. As above, we split into subcases
 \begin{itemize}
 \item $N_{1}\geq N$
 \item $N_{1}\leq N$
 \end{itemize}
 In the first subcase, one estimates
 \begin{equation}
 \begin{aligned}
 &N^{2-2s}\sum _{N_3 \sim N_{2}\geq N_{1},N_{4}}\int N_{1}^{2s} a_1 F_{2}F_{3}a_{4} \\
 \lesssim &N^{2-2s}\sum_{N_3 \sim N_{2}\geq N_{1}, N_{4}}\|F_{2}F_{3}\|_{L_{x,t}^{2}}N_{1}^{s}\|a_1\|_S N_{4}^{s}\|a_{4}\|_S \min \left(  (\frac{N_{1}}{N_{4}})^{s}, (\frac{N_{4}}{N_{1}})^{1/2-s}\right).
  \end{aligned}
 \end{equation}
As in the first case of $A_I$, we handle the double sum in  $N_{1}, N_{4}$, and we may estimate this expression by
 \begin{equation}
 \begin{aligned}
&N^{2-2s}\sum_{N_3 \sim N_{2}}\|F_{2}F_{3 }\|_{L_{t,x}^{2}}\\
& \lesssim N^{2-2s}\sum_{N_{2}}|I|^{1/2-}\|F_{2}\|_{L^{\infty-}_t L^4_x}\|F_{3}\|_{L^{\infty-}_t L^4_x}  \\
& \lesssim N^{2-2s}\lambda(t_{k})^{1-}
 \end{aligned}
 \end{equation}
 Finally, the case $N_{1}\leq N$, proceeds analogously with an extra $N^{s}$ loss, which is allowable.
\subsection*{Three random pieces}
As mentioned in our discussion of the cancellation of three random terms above, we only need to control 
\begin{equation}
\int_{I}\int \overline \Delta I_{N}a I_{N}(|F|^{2}F)
\end{equation}
Here, we recall \eqref{eq: fgood} and bound the above by
\begin{equation}
\|I_{N}a\|_{X^{1,b}[I]}\|I_{N}(|F|^{2}F)\|_{X^{1,1-b}[I]}\lesssim \|I_{N}a\|_{X^{1,b}[I]} N^{1-s}
\end{equation}
Note we use that $I_{N}$ can gain $1-s$ derivative by losing $N^{1-s}$. Now, plugging in \eqref{eq: directmodify}, we bound the above via
\begin{equation}
 N(T)^{1-s}(\frac{N(T)}{N(t_{k})})^{1-s}(\frac{1}{\lambda(t_{k})})^{1+ \epsilon_{1}}
\end{equation}
Summing over all LWP intervals yields 
\[
N(T)^{1-s}\frac{1}{\lambda(t)^{1+C\epsilon_{2}}},
\]
which is acceptable provided $0 < \delta \ll s$, and $0 < \epsilon_2 \ll 1$.

\subsection*{Estimates of $A_{II}$}

We recall the expression for $A_{II}$:
\[
A_{II} := \Re \int \overline{I_N \Delta a} \left[  |I_N( a + F) |^2 I_N (a + F))  - |I_N a|^2 I_N a \right] .
\]
Also recalling again our discussion on the cancellation of the three random terms, we note there will be no need to consider the three random pieces case here. In light of Remark \ref{rem: wdww}, we will work on $I=[\tau_{k}^{j}, \tau_{k}^{j+1}]$, and prove estimate on this interval.
\subsection*{One or two random pieces}
We may combine the estimates for one or two random pieces here. As above, we let $N_2 \geq N_3 \geq N_4$. Once again, we treat the case where the random piece is at the highest allowable frequency. We consider the cases:
\begin{enumerate}
\item $N_{1}\sim N_{2}\geq N_{3}\geq N_{4}$
\begin{itemize}
\item $N_{1}\geq N$
\item $N_{1}\leq N$
\end{itemize}
\item $N_{1}\ll N_{2}$, in this case one must has $N_{2}\sim N_{3}$
\begin{itemize}
\item $N_{1}\geq N$, note that in this case one must have $N_{2}\geq N$
\item $N_{1}\leq N$.
\end{itemize}
\end{enumerate}

Recall that
\begin{equation}
\|\nabla Ia\|_{S}\lesssim (\frac{N}{N(t_{k})})^{1-s}\frac{1}{\lambda(t_{k})}
\end{equation}
and 
\begin{equation}
\|a\|_{X^{s,b}}\sim \|\nabla^{s}a\|_{S}\sim \frac{1}{\lambda(t_{k})^{s}}.
\end{equation}
We start with subcase $N_{1}\sim N_{2}\geq N_{3}\geq N_{4}$, $N_{1}\geq N$. In this case, we estimate
\begin{equation}\label{eq: a2e1}
\begin{aligned}
&\sum_{N_2 \sim N_{1}\geq N_{3}\geq N_{4}}\int N_{1}^{2}Ia_{1}\frac{N^{1-s}}{N_{1}^{1-s}}F_{2}Ia_{3}Ia_{4}\\
&\lesssim N^{1-s}\sum_{N_{1} \sim N_2 \geq N_{3}\geq N_{4}}N_{1}N_{1}^{s}\|Ia_{1}Ia_{4}\|_{L_{x,t}^{2}}\|F_{2}Ia_{3}\|_{L_{t,x}^{2}}.
\end{aligned}
\end{equation}
Using bilinear Strichartz estimates for the $a_1, a_{4}$, term, and the random data control for $F_{2}$, we derive 
\begin{equation}
\begin{aligned}
\lesssim &N^{1-s}\sum_{N_{1} \sim N_2\geq N_{3}\geq N_{4}}N_{1}N_{1}^{s}(\frac{N_{4}}{N_{1}})^{1/2}\|Ia_{1}\|_{S}\|a_{4}\|_{S}\|F_{2}\|_{L_{t,x}^{4}} \|
a_{3}\|_{S}\\
=&N^{1-s}\lambda(t_{k})^{\frac{1}{2} - }\sum_{N_{1} \sim N_2\geq N_{3}\geq N_{4}}N_{1}\|Ia_1\|_{S}\|F_{2}\|_{L^{\infty-}_t L^4_x} \|a_{3}\|_{S}(\frac{N_{4}}{N_{1}})^{1/2-s}N_{4}^{s}\|a_{N_{4}}\|_{S}\\
\lesssim&
 N^{1-s}\lambda(t_{k})^{\frac{1}{2} - }\sum_{N_{1}\geq N_{3}}N_{1}\|Ia_1\|_{S}\|F_{2}\|_{L^{\infty-}_t L^4_x}\|a_{3}\|_{S}\|a_{3}\|_{X^{s,b}}(\frac{N_{3}}{N_{1}})^{1/2-s}\\
 \lesssim &N^{1-s}\lambda(t_{k})^{\frac{1}{2} - }\frac{1}{\lambda(t_{k})^{s}}\sum_{N_{1}\geq N_{3}}N_{1}\|Ia_1\|_{S}\|F_{2}\|_{L^{\infty-}_t L^4_x}\|a_{3}\|_{S}(\frac{N_{3}}{N_{1}})^{1/2-s}\\
 \lesssim &N^{1-s}\lambda(t_{k})^{\frac{1}{2} - }\frac{1}{\lambda(t_{k})^{s}}\|\nabla Ia\|_{S_{0}}\\
 \lesssim &N^{1-s}\lambda(t_{k})^{\frac{1}{2} - } \frac{1}{\lambda(t_{k})^{s}}\frac{N^{1-s}}{N(t_{k})^{1-s}}\frac{1}{\lambda(t_{k})}.
\end{aligned}
\end{equation}
This is desirable. It should be remarked we don't use any regularity of $a_{N_{3}}$, so the above arguments also works  when $a_{N_{3}}$ is replaced by $F_{N_{3}}$.

Now, we go to  the subcase  $N_{1}\sim N_{2}\geq N_{3}\geq N_{4}$, $N_{1}\leq N$, then the $I$ operator is just the identity map. We estimate 
\begin{equation}
\begin{aligned}
&\sum_{N\geq N_{1} \sim N_2 \geq N_{3}\geq N_{4}}\int N_{1}^{2}a_1F_{2}a_{3}a_{4}\\
&\lesssim \lambda(t_{k})^{\frac{1}{2} -} \sum_{N_{1} \sim N_2 \geq N_{3}\geq N_{4}}N_{1}^{2}\|a_1\|_{S}\|F_{2}\|_{L_{t}^{- \infty} L^4_{x}}\|a_{3}\|_{S}\|a_{4}\|_{S}(\frac{N_{4}}{N_{1}})^{1/2}\\
&\lesssim N^{2-s}\lambda(t_{k}) \frac{1}{\lambda(t_{k})^{s}}.
\end{aligned}
\end{equation}
We note that while we need $s$ derivative of $a_1$, no regularity of $a_{3}$ is used and thus this argument applies equally to the case with two random pieces.

Now, let us go to subcase $N_{1}\ll N_{2}$, then one must have $N_{2}\sim N_{3}$.  We first consider the subcase $N_{1}\geq N$, then necessarily one has $N_{2}\geq N$. We observe that
\begin{equation}
\|I_{N}F_{2}\|_{L_{t,x}^{\infty}}\lesssim (N_{2})^{s-1}N^{1-s}\|F_{2}\|_{L_{t,x}^{\infty}}.
\end{equation}
One may estimates 
\begin{equation}\label{eq: t1}
\begin{aligned}
&\sum_{N_{2} \sim N_3 \geq N\geq N_{1}, N_{4}}N^{2-2s}\frac{N_{1}^{2-2s}}{N_{2}^{2-2s}}\int N_{1}^{2s}a_1F_{2}a_{3}a_{4}\\
&\lesssim \lambda(t_{k})\|F_{2}\|_{L_{t,x}^{\infty}}\|a_{3}\|_{S}(\frac{N_{1}}{N_{2}})^{2-2s}N_{1}^{2s}\|a_1a_{4}\|_{L_{t,x}^{2}},
\end{aligned}
\end{equation}
and we further estimate
\begin{equation}
N^{2s}_{1}\|a_1a_{4}\|_{L_{t,x}^{2}}\lesssim N_{1}^{s}N_{4}^{s}\|a_1\|_{S}\|a_{4}\|_{S}\min(\frac{N_{4}}{N_{1}})^{\frac{1}{2}-s}, (\frac{N_{1}}{N_{4}})^{s}.
\end{equation}
Plugging this back to \eqref{eq: t1}, one finish the estimate as $\lambda(t_{k})^{1-2s}N^{2-2s}$. 

Finally, we are left with the case $N_{1}\ll N_{2}$, $N_{3}\sim N_{2}$, $N_{1}\leq N$, one simply estimates this expression as 
\begin{equation}
\sum_{N_1, N_{2}\sim  N_{3}, N_{4}}N_{1}^{2}a_1F_{2}a_{3}a_{4}\lesssim N^{2-s}\lambda(t_{k})^{1-s},
\end{equation}
which is sufficient.

The estimate with two random terms follows as in the estimates for $A_I$.

\subsection*{Estimates for $B_I$ and $B_{II}$}

These estimates proceed similarly to the previous ones, but are somewhat simpler since we do not lose derivatives, and indeed it is easy to see when there is at least three random pieces appear in the estimate, the proof becomes more or less trivial. This is in sharp contrast compared to the case of $A_{I}$ and $A_{II}$.
 Also, purely deterministic case follows from the estimates in  \cite{colliander2009rough}. 

We recall
\begin{align*}
B_{I} &:= \Re \int \overline{I_N (| a |^2 a)}\left[ I_N (| a + F |^2 (a + F)) -  |I_N( a + F) |^2 I_N (a + F)  \right]  \\
B_{II} &:= \Re \int \overline{I_N (| a  |^2 a)} \left[  |I_N( a + F) |^2 I_N (a + F)) - |I_N a|^2 I_N a \right].
\end{align*}

\subsection*{One random piece}
We will record the estimates involving one random piece. Once again we let $N_1 \geq N_2 \geq N_3$ and $N_4 \geq N_5 \geq N_6$. We will see that $B_I$ and $B_{II}$ follow in a similar manner to $A_I$ and $A_{II}$ and hence we will sketch the estimates for $B_I$, and leave $B_{II}$ to an interested ready. We ignore complex conjugates as they will not feature in our argument.

We consider
\begin{align}
\int I_Na_1 I_N a_2 I_N a_3 \left[ I_N(F_4 a_5 a_6) - I_N F_4 I_N a_5 I_N a_6  \right],
\end{align}
and without loss of generality, assume that $N_i \geq 1$ for all $i$. As previously, we note that in order for this expression to be nonzero, we will need $N_4 \gtrsim N$. We let $N_{123}$ be the resulting frequency from the convolution of the first three terms. In this setting, we need to consider two cases:

\begin{itemize}
\item $N_{1} \gtrsim N_4$,
\item $N_{1} \ll N_4$.
\end{itemize}

We consider the first case, in which we use Bernstein and H\"older's inequality and estimate
\begin{align}
&\left| \sum_{N_4 \sim N_1 \geq N_2, N_3, N_5, N_6}  \iint \left( \frac{N^{1-s}}{N_{1}^{1-s} } a_1 \right) a_2a_3  \left(\frac{N^{1-s}}{N_4^{1-s}}F_4 \right) a_5 a_6 \right| \\
& \simeq N^{2-2s} \sum_{N_4 \sim N_1 \geq N_2, N_3, N_5, N_6} N_{1}^{s-1} N_4^{s-1} \|a_1 a_2 a_3 a_5\|_{L^2_{t,x}} \|F_4 a_6\|_{L^2_{t,x}} \\
& \simeq N^{2-2s} \sum_{N_4 \sim N_1 \geq N_2, N_3, N_5, N_6} N_{1}^{2s-2+ \frac{1}{2}} \|a_1 a_2 a_3 a_5\|_{L^2_{t} L^1_x} \|F_4 a_6\|_{L^2_{t,x}} \\
& \simeq N^{2-2s} \sum_{N_4 \sim N_1 \geq N_2, N_3, N_5, N_6} N_{1}^{2s-2+\frac{3}{2}}  \|a_2\|_{L^\infty_t L^2_x} \|a_3\|_{L^\infty_t L^2_x} \|a_1 a_5\|_{L^2_{t,x}} \|F_4 a_6\|_{L^2_{t,x}} .
\end{align}
Since $2s - \frac{1}{2} < 2s$, we can estimate this as in the kinetic term.

In the second case, if $N_1 \ll N_4$, then since the convolution of the first three terms and the convolution of the last three terms are paired, we must have $N_4 \sim N_5$, and we can further estimate based on whether
\begin{itemize}
\item $N_{1} \geq N$,
\item $N_{1} \leq N$.
\end{itemize}
as before.

\subsection*{Two random pieces}
There are two subcases we will consider
\begin{itemize}
\item $N_{1} \gtrsim N_4$
\item $N_{1} \ll N_4$.
\end{itemize}
We again estimate mimicking the kinetic term estimates, to obtain
\begin{align}
&\left| \sum_{N_4 \sim N_1 \geq N_2, N_3, N_5, N_6}  \iint \left( \frac{N^{1-s}}{N_{1}^{1-s} } a_1 \right) a_2a_3  \left(\frac{N^{1-s}}{N_4^{1-s}}F_4 \right) F_5 a_6 \right| \\
& \simeq N^{2-2s} \sum_{N_4 \sim N_1 \geq N_2, N_3, N_5, N_6} N_{1}^{s-1} N_4^{s-1} \|a_1 a_2 a_3 a_6\|_{L^2_{t,x}} \|F_4 F_5\|_{L^2_{t,x}} \\
& \simeq N^{2-2s} \sum_{N_4 \sim N_1 \geq N_2, N_3, N_5, N_6} N_{1}^{2s-2+ \frac{1}{2}} \|a_1 a_2 a_3 a_6\|_{L^2_{t} L^1_x} \|F_4 F_5\|_{L^2_{t,x}} 
\end{align}
and again we can use Bernstein on $a_2, a_3$.

In the second case, if $N_1 \ll N_4$, then since the convolution of the first three terms and the convolution of the last three terms are paired, we must have $N_4 \sim N_5$, and we can further estimate based on whether
\begin{itemize}
\item $N_{1} \geq N$,
\item $N_{1} \leq N$,
\end{itemize}
and again we can estimate as with the kinetic terms. 

\subsection*{Estimate for momentum}
These estimates proceed via direct computation, and we refer as well to the explanation in \cite{colliander2009rough}, the proof of \eqref{eq: pm} is similar to the proof for kinetic part of \eqref{eq: pm}. 
\end{proof}

\section{Proof of the bootstrap lemma and the main theorem}\label{sec: boot}
In this section we establish the main bootstrap argument, Lemma \ref{lem: boot}, as well as the proof of the main Theorem \ref{thm: mainrigor}. First, we recall our ansatz \eqref{eq: ans}:
\begin{align}
u(t,x)&=a(t,x)+F(t,x),\\
a(t,x)&=\frac{1}{\lambda(t)}(\qb+\epsilon)(\frac{x-x(t)}{\lambda(t)})e^{-i\gamma(t)},
\end{align}
where $a$ plays the role of the full solution $u$ in \cite{colliander2009rough}, and satisfies the forced NLS \eqref{eq: eqfora}, and where the parameters $\lambda(t), x(t), b(t), \gamma(t)$ are chosen so that the orthogonality conditions \eqref{eq: modoth1} -- \eqref{eq: modoth4} hold. 
Having establishing the desired energy estimates for $E(I_{N}a)$ and $P(I_{N}a)$ under the bootstrap assumptions of Lemma \ref{lem: boot}, the proof of Lemma \ref{lem: boot} essentially follows as in \cite[Section 4]{colliander2009rough}, with some changes in our current setting which we highlight below. In particular, we will verify that given our estimates on $E(I_{N}a)$ and $P(I_{N}a)$, the key computations in \cite{merle2003sharp, merle2004universality, merle2006sharp, merle2005blow} still hold following the bootstrap scheme in \cite{planchon2007existence}. 

It should be noted that unlike the full solution $u$, the nonlinear component of the solution $a$ does not satisfy an exact mass conservation law, which adds additional technical difficulties in the last step of Section \ref{sub: eekll} below.

\subsection{Energy estimates imply persistence of log-log regime}\label{sub: eekll}
\subsection*{Step 1} We use the rescaled the time variable $s$, where $ds=\lambda^{-2}dt$, and we set $t(s_{0})=0$ and $t(s_{+})=T$. We use the forced NLS \eqref{eq: eqfora} to derive 
\begin{equation}\label{eq: e1}
\begin{aligned}
\partial_{s}\Sigma_{b}+\partial_{s}\epsilon_{1}-M_{-}(\epsilon)+b\Lambda \epsilon_{1} &=
\bigl(\frac{\lambda_{s}}{\lambda}+b\bigr)\Lambda \Sigma_{b}+\tilde{\gamma}_{s}\Theta_{b}+\frac{x_{s}}{\lambda}\nabla \Sigma_{b} +\llsb\Lambda \epsilon_{1} \\
& \hspace{13mm}+\tgs\epsilon_{2}+\xls \nabla \epsilon_{1}+\Im \Psi_{b}-R_{2}(\epsilon)-G_{2},
\end{aligned}
\end{equation}
\begin{equation}\label{eq: e2}
\begin{aligned}
\partial_{s}\tb+\partial_{s}\epsilon_{2}+M_{+}+b\epsilon_{2}=&\llsb\Lambda \tb-\tgs\sbb+\xls\nabla \tb\\
&\hspace{6mm}+ \lls b\Lambda \epsilon_{2}-\tilde{\gamma}_{s}\epsilon_{1}+\xls \nabla \epsilon_{2}-\Re \Psi_b+R_{1}(\epsilon)+G_{1}.
\end{aligned}
\end{equation}
where $\tilde{\gamma}=-s-\gamma$, and  $M_{+}, M_{-}$, and $R_{1}, R_{2}$ are defined via
\begin{equation}
|\qbb+\epsilon|^{2}(\qbb+\epsilon)-|\qbb|^{2}\qb=M_{+}(\epsilon)+iM_{-}(\epsilon)+R_{1}(\epsilon)+iR_{2}(\epsilon),
\end{equation}
i.e. $M_{\pm}$ picks up the first order term (w.r.t to $\epsilon$), and $R_{1}+iR_{2}$ picks up the second and higher order term (in $\epsilon$).

And $G=G_{1}+iG_{2}$ is defined via
\begin{equation}\label{eq: extrag}
G(t,x)=-\bigl(|\qbb+\epsilon+\tilde{F}|(\qb+\tilde{F})-|\qbb+\epsilon|^{2}(\qb+\epsilon)\bigr),
\end{equation}
where $\tilde{F}(t,x)=\lambda(t)F(t,\lambda(t)x+x(t))e^{i\gamma(t)}$.

Note that \eqref{eq: e1} and \eqref{eq: e2} are exactly equations (4.2) and (4.3) in \cite{colliander2009rough} except that we have two extra terms, $G_{1}$ and $G_{2}$. We will see that these terms can be treated perturbatively due to the fact that $F$ is the linear  evolution of randomized initial data.

\subsection*{Step 2} We now derive some preliminary estimates using (almost) conservation laws, and modulation estimates. In particular, using our control of $E(I_{N}a)$ and $P(I_{N}a)$ obtained in the previous section, we derive the following result.
\begin{lem}\label{lem: conser}
For all $s\in [s_{0},s+]$,
\begin{align}
\biggl|2(\epsilon_1, \sbb &+b\Lambda \tb-\Re \Psi_{b}+2(\epsilon_{2}, \tb-b\Lambda \sbb-\Im \Psi_{b})) \nonumber\\
 \hspace{14mm} &-2\left(2\Xi+\int |I_{N\lambda}\nabla \epsilon|^{2}-3Q^{2}I_{N\lambda}\epsilon_{1}^{2}-\int Q^{2}I_{N\lambda \epsilon_{2}}^{2}\right) \biggr | \nonumber \\
& \leq \delta_{0}\left(\errh+\errm\right)+\Gamma_{b}^{1-C\eta}, \\
|(\epsilon_{2}, \nabla Q)|& \leq \delta_{0}(\errh)^{1/2}+\Gamma_{b}^{10}.
\end{align}
\end{lem}
Here $\delta_{0} > 0$ is some small constant. This step is exactly same as the derivation of (4.5) and (4.6) in \cite{colliander2009rough}.  In this step, we rely on the bootstrap assumption \eqref{eq: balambdasmall}, and almost conservation law Proposition \ref{pro: almost}.

\medskip
By substituting \eqref{eq: modoth1}--\eqref{eq: modoth4} into \eqref{eq: e1} and \eqref{eq: e2}, we derive the following standard modulation estimate.
\begin{lem}\label{lem: mode}
For  $s\in [s_{0},s+]$
\begin{equation}\label{eq: mode1}
\biggl|\llsb\biggr|+|b_{s}|+|x_{s}|\lesssim \Xi (s)+\errh+\errm+\Gamma_{b}^{1-C\eta}+\FFF(s),
\end{equation}
\begin{equation}\label{eq: mode2}
\biggl|\tgs-\frac{(\epsilon_{1}, L_{+}\Lambda^{2} Q)}{\|\Lambda Q\|_{L_{x}^{2}}^{2}}\biggr|\lesssim\Gamma_{b}^{1-C\eta}+\FFF(s)
\end{equation}
where $\delta_{0} > 0$ is a small constant and $\FFF(s)\geq 0$ satisfies
\begin{equation}\label{eq: okerror}
\int_{s}^{s+}\FFF(s)ds\lesssim \lambda(s)^{\alpha_{2}},\qquad \forall s\in [s_{0}, s_{+}]
\end{equation}
where $\alpha_{2}>0$
\end{lem}
\begin{rem}
The extra term $\FFF$, is completely perturbative, though it is only estimated in time average sense, however as this term appears when estimating the time derivative of the modulation parameters, this is sufficient. Heuristically, point-wise,
\[
\FFF(s)\sim -\partial_{s}\lambda(s)^{\alpha_{2}}\sim -\frac{\lambda_{s}}{\lambda}\lambda^{\alpha_{2}}\sim b\lambda^{\alpha_{2}}\ll \Gamma_{b}^{100}.
\]
\end{rem}

Lemma \ref{lem: mode} should be compared to (4.7) and (4.8) in \cite{colliander2009rough} (in the $H^{s}$ setting). Compared to the standard modulation estimates in the $H^{1}$ setting, the term $\FFF(s)$ is introduced since to account for the cut-off $I_{N\lambda}$ in the estimate.  

In our setting, we need to verify that the extra term in \eqref{eq: e1} and \eqref{eq: e2} is also perturbative. To see this, we briefly recall how modulation estimate is done. To derive \eqref{eq: mode1} and \eqref{eq: mode2}, one substitutes the four orthogonality  conditions \eqref{eq: modoth1}-\eqref{eq: modoth4} into \eqref{eq: e1} and \eqref{eq: e2} to cancel the $\partial_{s}\epsilon_{1}, \partial_s \epsilon_{2}$ terms. For example, to substitute \eqref{eq: modoth1} into \eqref{eq: e1} and \eqref{eq: e2}, one needs to take the $L_{x}^{2}$ inner product of \eqref{eq: e1} and $y^{2}\qbb$, and the $L_{x}^{2}$ inner product of \eqref{eq: e2} and $y^{2}\theta_{b}$, respectively, and sum up. Compared to \cite{colliander2009rough}, we obtain extra terms resulting from $G_{1}, G_{2}$, which satisfy
\begin{align}\label{eq: gerr}
(|y|^{2}|\qbb|, |G|) &\lesssim \int |y|^{2}|\qbb|\bigl(|\qbb+\epsilon|^{2}|\tilde{F}|+|\qbb+\epsilon||\tilde{F}|^{2}\bigr).
\end{align} 
We then claim that for any $s_{1}\in [s_{0},s_{+})$, we have
\begin{equation}\label{eq: absorb}
\int_{s_{1}}^{s_{+}} \int |y|^{2}|\qbb|(|\qbb+\epsilon|^{2}|\tilde{F}|+|\qbb+\epsilon||\tilde{F}|^{2})  \lesssim \lambda^{\alpha_{2}}(s_{1}),\quad \text{ for some } \alpha_{2}>0,
\end{equation}
and thus, we may absorb these extra terms into the $\FFF$ which satisfies \eqref{eq: okerror}. To establish this bound, we proceed as follows: let $t(s_{1})\in [t_{k_{1}}, t_{k_{1}+1}]$ and $T_{+}=s_{+}$, and $\lambda(T_{+})\sim 2^{-k_{+}}$. We can split $[t_{k_{1}},T_{+})$ into disjoint intervals $\{ I_{k}\}_{k_1}^{k+}$, and we may split every $I_{k}$ into disjoint LWP intervals $I_{k}^{j}=[\tau_{k}^{j}, \tau_{k}^{j+1}]$ such that $|I_{k}^{j}|\sim \lambda(t_{k})^{-2}$. Recall that for any $k$, there exists at most $k$ such intervals, via bootstrap assumption \eqref{eq: balwpinteval}.
Now, we may estimate the LHS of \eqref{eq: absorb}, in the original non-rescaled variable, as
\begin{equation}\label{eq: absorb2}
\begin{aligned}
&\int_{s_{1}}^{s_{+}} \int |y|^{2}|\qbb|(|\qbb+\epsilon|^{2}\tilde{F}+|\qbb+\epsilon||\tilde{F}|^{2})\\
& \lesssim \sum_{t_{k_{1}}}^{T_{+}} \int \frac{1}{\lambda(t)^{1/2}}(\|a\|_{L_{x}^{4}}^{2}\|F\|_{L_{x}^{4}}+\|a\|_{L_{x}^{4}}\|F\|_{L_{x}^{4}}^{2} + \|F\|_{L_{t,x}^{4}}^3)dt\\
\lesssim &\sum_{k=k_{1}}^{k_{+1}}\sum_{j}\left(\|a\|_{L_{t,x}^{4}[I_{k}^{j}]}^{2}\|F\|_{L_{t,x}^{4}I_{k}^{j}}+\|
a\|_{L_{t,x}^{4}[I_{k}^{j}]}\|F\|_{L_{t,x}^{4}[I_{k}^{j}]}^{2} + \|F\|_{L_{t,x}^{4}[I_{k}^{j}]}^3 \right).
\end{aligned}
\end{equation} 
Note that up to an exceptional set of small probability (depending on $p$), one has 
\begin{equation}
\|F(t,x)\|_{L_{t,x}^{p}}\lesssim 1,
\end{equation}
which, combined with the estimate $\|F(t,x)\|_{L_{t}^{\infty}L_{x}^{2}}\lesssim 1$,  gives
\begin{equation}
\|F(t,x)\|_{L_{t}^{4}L_{x}^{4}[I_{k}^{j}]}\lesssim |I_{k}^{j}|^{\alpha_{p}}, \text{ where } \lim_{p\rightarrow \infty}\alpha_{p}=\frac{1}{4},
\end{equation}
By the standard local theory\footnote{Here we can simply apply the usual deterministic $L_{x}^{2}$ local theory rather than the modified probabilistic version in the current article.}
\begin{equation}
\|a\|_{L_{t,x}^{4}[I_{k}^{j}]}\lesssim 1,
\end{equation}
hence we can choose $p$ large enough, and estimate \eqref{eq: absorb2} by 
\begin{equation}
\sum_{k=k_{1}}^{k_{+}}k 2^{-2k\alpha_{p}}\lesssim \lambda(s_{1})^{(-2\alpha_{p})+},
\end{equation}
which establishes \eqref{eq: absorb}, and consequently, Lemma \ref{lem: mode}. We will repeatedly rely on the above argument to handle the extra terms caused by $G_{1}, G_{2}$, we do not repeat the details.

\subsection*{Step 3} This step mainly concerns the derivation of the (local) virial estimate, as well as its sharpening via the tail term $\zbb$. This is the core part of the Merle-Rapha\"el \cite{} log-log analysis.  The key point here, similar to \cite{colliander2009rough}, is to make sure the original Merle-Rapha\"el computation remains valid by showing all extra terms introduced are perturbative.

One has following virial estimates. \footnote{Estimate \eqref{eq: lv} was called as local virial estimate in \cite{merle2003sharp}, and the global virial estimate in \cite{colliander2009rough}.}
\begin{lem}\label{lem: localvirial}
There exists $c_0 > 0$ such that for all $s \in [s_0, s_+)$, so that
\begin{equation}\label{eq: lv2}
b_{s}\geq c_{0}(\Xi(s)+\errh+\errm)-\Gamma_{b}^{1-C\eta}-\FFF(s),
\end{equation}
where $\FFF$ satisfy \eqref{eq: okerror}.
\end{lem}

\begin{lem}\label{lem: movirial} 
Let
\begin{equation}
f_{1}:=\frac{b}{4}\|y\qbb\|_{2}^{2}+\frac{1}{2}\Im \int y\nabla \zbb\overline{\zbb}+(\epsilon_{2},\Lambda \Re \zbb)-(\epsilon_{1},\Lambda \Im \zbb).
\end{equation}
Then there holds for a universal constant $c_1$ such that for all $s \in [s_0, s_+)$ that
\begin{equation}\label{eq: movirial}
\partial_{s}f_{1}(s)\geq c_{1}(\Xi(s)+\errh+\errm+\Gamma_{b})-\frac{1}{\delta_{1}}\int_{A\leq |x|\leq 2A}|\xi|^{2}-\FFF(s),
\end{equation}
where $\FFF$ satisfies \eqref{eq: okerror}.
\end{lem}
\begin{rem}
In the previous lemma, one should think of $f_1$ as a modified version of $b$, in particular satisfying  $f_{1}\sim \frac{1}{4}\|yQ\|_{2}^{2}b$. 
\end{rem}

Lemma \ref{lem: localvirial} and Lemma \ref{lem: movirial} should be compared to  Lemma 4.3 and Lemma 4.4, respectively, in \cite{colliander2009rough}. We can again use the argument from Lemma \ref{lem: mode} above to argue that the extra terms created by $G_{1}$ and $G_{2}$ in \eqref{eq: e1} and \eqref{eq: e2} can also be absorbed into the error $\FFF$. For example, to derive Lemma \ref{lem: localvirial}, one computes the $L_{x}^{2}$ inner product of $-\Lambda \Theta_{b}$ and \eqref{eq: e1}, and the $L_{x}^{2}$ inner product of $\Lambda \Sigma_{b}$ and  \eqref{eq: e2}, and sums them together, substituting into \eqref{eq: modoth3}. Ultimately, the extra terms caused by the $G_{1}, G_{2}$ are controlled by 
\begin{equation}
(\Lambda\qbb|, |G|)\lesssim \int |\Lambda\qbb|(|\qbb+\epsilon|^{2}|\tilde{F}|+|\qbb+\epsilon||\tilde{F}|^{2}),
\end{equation}
which can be handled similarly to \eqref{eq: gerr} above. We omit the details.

\subsection*{Step 4} In this step, we need to control the $L^{2}_x(\RRR^2)$ dispersion at infinity. Recall that 
\[
A=A_{b}=e^{\frac{a\pi}{b}},
\]
and $\Psi$ is a radial cut-off function, with $\Psi=0$, for $|x|\leq 1/2$ and $\Psi=1$ for $|x|\geq 3$.  Let $\Psi_{A}(x)=\Psi(\frac{x}{A})$. And one has 
\begin{lem}\label{lem: L2disper}
There holds for some universal constants $C, c_3 > 0$ so that for all $s \in [s_0, s_+)$, it holds that
\begin{equation}\label{eq: l2dis}
\partial_{s}\int \Psi_{A}|\epsilon|^{2}\geq c_{3}b\int_{A\leq |x|\leq 2A}|\epsilon|^{2}-\Gamma_{b}^{a/2}\int |\nabla I_{N\lambda}\epsilon|^{2}-\Gamma_{b}^{1+Ca}-\FFF(s)-\partial_{s}\HHH(s),
\end{equation}
where $\FFF$ sastisfies \eqref{eq: okerror}, and $\HHH$ satisfies the estimate
\begin{equation}\label{eq: absorb3}
|\HHH(s)|\lesssim \lambda^{\alpha_{3}}(s), \text{ for some } \alpha_{3}>0.
\end{equation}
\end{lem}
Lemma \ref{lem: L2disper} corresponds to Lemma 4.5 in \cite{colliander2009rough}, up to certain technical modifications. We quickly go over its proof, focusing only on what differs compared to the proof of \cite[Lemma 4.5]{colliander2009rough}. While one could actually absorb the $\FFF$ term into $\partial_{s}\HHH$, we choose to proceed in a manner that more closely follows the original presentation of \cite{colliander2009rough}.
\begin{rem}\label{rem: L2disper}
We recall the tail introduced in the previous section, $\zbb$ and set $\tilde{\epsilon}=\epsilon-\zbb$. Note that $\|\nabla \epsilon-\nabla \tilde{\epsilon}\|_{L_{x}^{2}}=\|\nabla \zbb\|_{L_{x}^{2}}\lesssim \Gamma_{b}^{1-C\eta}$, which implies, by choosing $a\gg C\eta$, that
\begin{equation}
\|\nabla I_{N\lambda}\epsilon-\nabla I_{N\lambda}\tilde{\epsilon}\|_{2}\lesssim \Gamma_{b}^{1+Ca},
\end{equation}
thus Lemma \ref{lem: L2disper} implies, in particular, that
\begin{equation}\label{eq: effectivel2}
\partial_{s}\int \Psi_{A}|\epsilon|^{2}\geq c_{3}b\int_{A\leq |x|\leq 2A}|\epsilon|^{2}-\Gamma_{b}^{1+Ca}-\Gamma_{b}^{a/2}\int |\nabla I_{N\lambda}\tilde{\epsilon}|^{2}-\FFF(s)-\partial_{s}\HHH(s)
\end{equation}
\end{rem}

\begin{proof}[Proof of Lemma \ref{lem: L2disper}]
Recall since $\qbb$ is supported in $|x|\lesssim \frac{1}{b}$, thus one has 
\begin{equation}\label{eq: sp}
\Psi_{A}|\qbb|^{2}\equiv 0.
\end{equation}
Thus
\begin{equation}
\Psi_{A}(|\inl \epsilon|^{2}-|\inl \epsilon+\qbb|^{2})=0.
\end{equation}
Now note
\begin{equation}\label{eq: tran}
\begin{aligned}
\Psi_{A}|\epsilon|^{2}
=&\Psi_{A}(|\epsilon|^{2}-|I_{N\lambda}\epsilon|^{2})\\
+&\Psi_{A}(|\inl \epsilon|^{2}-|\inl \epsilon+\qbb|^{2})\\
+&\Psi_{A}(-|\inl (\epsilon+\qbb)|^{2}+|\inl \epsilon+\qbb|^{2})\\
+&\Psi_{A}(|\inl \epsilon+\qbb|^{2}).
\end{aligned}
\end{equation}
Observe that the second line of \eqref{eq: tran} is $=0$ thanks to \eqref{eq: sp}. 
Let $\HHH$ be defined as 
\begin{equation}
\HHH=\Psi_{A}(|\epsilon|^{2}-|I_{N\lambda}\epsilon|^{2})+\Psi_{A}(-|\inl (\epsilon+\qbb)|^{2}+|\inl \epsilon+\qbb|^{2}).
\end{equation} 
$\HHH$ satisfies \eqref{eq: absorb3} since $\epsilon$ is bounded in $H^{s}$ due to the bootstrap assumption \eqref{eq: bah1control}, $\qbb$ is a nice function (uniformly in $b$), and $I_{N\lambda}-Id$ removes all frequencies above $N(t)\lambda(t)\sim \lambda(t)^{-\delta}$ for some $\delta>0$.

Moreover, we have
\begin{equation}
\frac{d}{ds}\int \Psi(\frac{x-x(t)}{A\lambda(t)})|\epsilon|^{2}=\frac{d}{ds}\frac{x-x(t)}{A(t)\lambda(t)}|I_{N\lambda}a|^{2}-\partial_{s}\HHH,
\end{equation} 
where we recall the ansatz for $a$ given by 
\[
a=\frac{1}{\lambda(t)}(\qbb+\epsilon)\left(\frac{x-x(t)}{\lambda(t)}\right)e^{i\gamma(t)}
\]
and that the scaling in $\lambda$ is $L_{x}^{2}$ invariant.

As mentioned above, the role $a$ plays for us is the same role played by $u$ in the proof of \cite[ Lemma 4.5]{colliander2009rough}. One may follow the computations leading to \cite[(4.27)]{colliander2009rough} and the formula above (4.27) in \cite{colliander2009rough} to derive\footnote{In the original \cite[(4.27)]{colliander2009rough}, there should be a $2\lambda^{2}$ before the $\Im(\cdots)$ term. Additionally, on the LHS of  \cite[(4.27)]{colliander2009rough}, it should read $\Psi_{A\lambda}(x-x(t))$ rather than $\Psi_{A}$.}
\begin{equation}\label{eq: eqok}
\begin{aligned}
\frac{d}{ds}\int \Psi(\frac{x-x(t)}{A\lambda(t)}) |I_{N \lambda} a|^2 \geq &c_{3}b\int_{A\leq |x|\leq 2A}|\epsilon^{2}|-\Gamma_{b}^{1+Ca}-\Gamma_{b}^{a/2}\int |\inl \nabla \epsilon|^{2}\\
&-\FFF_{1}(s)+2\lambda^{2}\Im\int \Psi(\frac{x-x(t)}{A\lambda(t)})\overline{I_{N}a}[I_{N}(a|a|^{2})-I_{N}a|I_{N}a|^{2}]\\
&-\lambda^{2}\bigl|\Psi(\frac{x-x(t)}{A\lambda(t)})\tilde{I}_{N}a\overline(I_{N})a\bigr|\\
&+2\lambda^{2}\Im \int  \Psi(\frac{x-x(t)}{A\lambda(t)})\overline(I_{N}a)\bigl[I_{N}\bigl(|a+F|^{2}(a+F) \bigr)-I_{N} \bigl(|a|^{2}a\bigr)\bigr]
\end{aligned}
\end{equation}
where $\FFF_{1}$ satisfies \eqref{eq: okerror} and consequently may be absorbed into $\FFF$. It has been explained in detail in \cite{colliander2009rough} why the second line and third line of \eqref{eq: eqok} can also be absorbed into $\FFF$. Note that in the final line of \eqref{eq: eqok}, we have used that $a$ satisfies a forced NLS \eqref{eq: eqfora}.

\medskip
We claim for all $s_{1}\in [s_{0}, s_{+}]$, one has 
\begin{equation}\label{eq: errorokagian}
\int_{s_{1}}^{s_{+}}\left|2\lambda^{2}\Im \int  \Psi(\frac{x-x(t)}{A\lambda(t)})\overline{(I_{N}a)}\biggl[I_{N}\bigl(|a+F|^{2}(a+F)\bigr) -I_{N}\bigl(|a|^{2}a\bigr)\biggr]\right |\lesssim \lambda(s_{1})^{\alpha_{2}}, 
\end{equation}
for some $\alpha_{2}>0$. This is again similar to \eqref{eq: absorb}, since if one lets $t(s_{1})\in [t_{k_{1}}, t_{k_{1}+1})$, and $T_{+}=t(s_{+})$, then using $\lambda^{2}ds=dt$, one can bound the LHS via 
\begin{equation}
\int_{t_{k_{1}}}^{T_{+}} (|a|_{L_{x}^{4}}^{3}+|F|_{L_{x}^{4}}^{3})\|F\|_{L_{x}^{4}},
\end{equation}
and proceed similarly as the proof of \eqref{eq: absorb}. We leave the details to interested readers.
\end{proof}

\subsection*{Step 5} In this step, we use the mass ``conservation'' law to combine Lemma~\ref{lem: movirial} and Lemma~\ref{lem: L2disper}, and derive Lyapunov type control.  For this part, we mostly directly referr to \cite{merle2006sharp} in \cite{colliander2009rough}. It should be noted, however, that unlike the $H^{1}$ case from  \cite{merle2006sharp}, or the $H^{s}, s>0$ case from \cite{colliander2009rough}, here we will need to handle an almost  conservation law rather than exact conservation law. Indeed, if $a$ solved the NLS, one would have
\[
\frac{d}{ds}\|\qbb+\epsilon\|_{2}^{2}=\frac{d}{ds}\|a\|_{L_{x}^{2}}\equiv 0.
\]
In our case, $a$ only solves a forced NLS \eqref{eq: eqfora} and thus does not enjoy precise mass conserv. We instead claim the following.
\begin{lem}\label{lem: slowmass}
$\forall s\in [s_{0}, s_{+})$, one has 
\begin{equation}\label{eq: slowmass}
\frac{d}{ds}\|\qbb+\epsilon\|_{L_{x}^{2}}^{2}=\partial_{s}\GGG,
\end{equation}
where $|\GGG(s)|\lesssim  \lambda^{\alpha_{2}}(s)$, for some $\alpha_{2}>0$.
\end{lem}
\begin{proof}
We compute
\begin{equation}
\frac{d}{ds}\|\qbb+\epsilon\|_{L_{x}^{2}}^{2}=\frac{d}{ds}\|a\|_{L_{x}^{2}}^{2}=\lambda^{2}\frac{d}{dt}\|a\|_{L_{x}^{2}}^{2},
\end{equation}
which is bounded by 
\begin{equation}
\lambda^{2}\left|\int a (|a+F|^{2}(a+F)-|a|^{2}a)\right|.
\end{equation}
We need only to verify that
\begin{align}
&\int_{s}^{s_{+}} \lambda^{2}\left|\int a (|a+F|^{2}(a+F)-|a|^{2}a)ds\right| \nonumber\\
&=\int_{t(s)}^{T_{+}}\left|\int a (|a+F|^{2}(a+F)-|a|^{2}a)dt \right| \nonumber\\
& \lesssim \lambda(s)^{\alpha_{2}}, 
\end{align}
for some $\alpha_{2}>0$. This is again similar to \eqref{eq: absorb} and we omit further details.
\end{proof}

Thus, one has by expanding \eqref{eq: slowmass} and observing that $Q$ is not dependent on $s$, that
\begin{equation}\label{eq: moremasscon}
\frac{d}{ds}\left(\|\qbb\|_{L_{x}^{2}}^{2}-\|Q\|_{L_{x}^{2}}^{2}+\|\epsilon\|_{2}^{2}+2(\epsilon_{1}, \Sigma_{b})+2(\epsilon_{2}, \Theta_{2})\right)=\partial_{s}\GGG.
\end{equation}
Now, we are ready to follow the computation in the proof of \cite[Proposition 4]{merle2006sharp}. We combine \eqref{eq: l2dis} and \eqref{eq: movirial} with the help of \eqref{eq: slowmass}, and, as in \cite{merle2006sharp}, we obtain
\begin{align*}
&\frac{1}{100}c_{3}\delta b f_{1}(s)+\frac{d}{ds}\int \Psi_{A}|\epsilon|^{2}\\
&\geq  \frac{1}{200}c_{3}\delta \Xi(s)+\frac{1}{200}c_{3}\delta b|\int I_{N\lambda}\nabla \epsilon|^{2}\\
& \hspace{14mm}+\frac{1}{200}c_{3}\delta b\int |\epsilon|^{2}e^{-|y|}+\frac{1}{200}c_{3}\delta b\Gamma_{b}-\FFF-\partial_{s}\HHH.
\end{align*}
We substitute
\begin{equation}
\frac{d}{ds}\int \Psi_{A}|\epsilon|^{2}=-\left(\|\qbb\|_{L_{x}^{2}}^{2}-\|Q\|_{L_{x}^{2}}^{2})+\|\epsilon\|_{2}^{2}+2(\epsilon_{1}, \Sigma_{b})+2(\epsilon_{2}, \Theta_{2})\right)-\partial_{s}\GGG,
\end{equation}
and let $\JJJ$ be defined as 
\begin{align}
\JJJ(s)&=\int |\qbb|^{2}-\int |Q|^{2}+2(\epsilon_{1}, \Sigma)+2(\epsilon_{2},\Theta)+\int (1-\Psi_{A}\epsilon^{2})\\
&\hspace{8mm}-\frac{1}{100}c_{3}\delta(\tilde{f}_{1}(b)-\int_{0}^{b}\tilde{f}_{1}(v)dv+b\{(\epsilon_{2}, \Lambda \Re\zbb)-(\epsilon_{1}, \Lambda \Im \zbb)\}-\HHH-\GGG,
\end{align}
where 
\[
\tilde{f}_{1}(b)=\frac{b}{4}\|y\qbb\|_{2}^{2}+\frac{1}{2}\Im \int y\nabla \zbb \overline{\zbb}.
\]
We then obtain
\begin{equation}\label{eq: ly}
\partial_{s}\JJJ\leq -Cb\left[\Gamma_{b}+\Xi+\int |\nabla \inl \tilde{\epsilon}|^{2}+\int |\epsilon(s)|^{2}e^{-|y|}+\int_{A\leq |x|\leq 2A}|\epsilon|^{2}\right]+\FFF,
\end{equation}
which corresponds to \cite[(4.28)]{colliander2009rough}, though the definition of $\JJJ$ now involves the correction $\HHH$ and $\GGG$.

\medskip

The main observation is the simple fact that $\JJJ$ is of size $b^{2}$, and the two extra correction terms are of size $\ll \Gamma_{b}^{100}$ and can be neglected. Hence, $\JJJ$ with the extra corrections $\HHH$ and $\GGG$ can still serve as Lyapunov functions as in \cite{colliander2009rough}.

Estimates  \eqref{eq: movirial} and \eqref{eq: ly} ensure the dynamics remain in the log-log blowup regime, and are enough to close the bootstrap lemma \ref{lem: boot}. Indeed, the rest of the proof of the bootstrap lemma follows almost line by line as in \cite{colliander2009rough}, as well as following from the original scheme in \cite{planchon2007existence}. We go over its proof quickly:

\begin{itemize}
\item One applies mass (almost) conservation law, Lemma \ref{lem: slowmass} to upgrade \eqref{eq: basmall} into \eqref{eq: besmall}. (In \cite{colliander2009rough}, one can just apply the exact mass conservation law).
\item One use the monotonicity of $\JJJ$ to upgrade \eqref{eq: bah1control} into \eqref{eq: beh1control}.
\item Estimate \eqref{eq: mode1} implies in average sense $\lambda_{s}/\lambda\sim -b$, this is already enough to upgrade \eqref{eq: bamono} into \eqref{eq: bamono}
\item Now the dynamics of $\lambda $ are dictated by the dynamics of $b$, and $b_{s}$ is governed by \eqref{eq: lv} and \eqref{eq: ly}. This will allows one to upgrade \eqref{eq: balambdasmall}, \eqref{eq: belambdasmall}, and to upgrade \eqref{eq: balwpinteval} into \eqref{eq: belwpinteval} .
\end{itemize}
This concludes the proof of bootstrap lemma \ref{lem: boot}. 

Finally, we may now prove the main theorem:

\begin{proof}[Proof of Theorem \ref{thm: mainrigor}]
Following from the probabilistic local wellposedness of Lemma \ref{lem: lwplocalpro} and the energy estimates of Proposition \ref{pro: almost}, as we have detailed above, one obtains an exception set of small probability so that the bootstrap lemma \ref{lem: boot} holds. As mentioned in the last two steps above, the dynamics of $\lambda$ are dictated by those of $b$, whose dynamics are controlled by \eqref{eq: ly} and \eqref{eq: lv}. This is sufficient to prove $\lambda(t)$ goes to zero as the desired rate, see \cite{merle2006sharp, colliander2009rough} for more details. This concludes the proof of Theorem \ref{thm: mainrigor}.
\end{proof}
\bibliographystyle{abbrv}
\bibliography{BG_2}
\end{document}